\documentclass[reqno,oneside]{amsart}

\address{Graduate School of Mathematics,
Nagoya University, 464-8602 Furo-cho, Chikusa, Nagoya, Japan} \email{ohta@math.nagoya-u.ac.jp}
\address{Graduate School of Mathematics,
Nagoya University, 464-8602 Furo-cho, Chikusa, Nagoya, Japan\\
Current Address: Department of Mathematics, Kyoto University, 606-8502 Kitashirakawa Oiwake-cho, Sakyo, Kyoto, Japan}
\email{sanda.fumihiko.5s@kyoto-u.ac.jp}

   \usepackage{amsthm}
   \usepackage{amsmath}
   \usepackage{amssymb}
   \usepackage{amsfonts}
   \usepackage{amscd}
   \usepackage{mathrsfs}  
   \usepackage{bm} 
   \usepackage{bbm}
   \usepackage{color}
   \usepackage[all]{xy}
   \makeatletter
    
    \@addtoreset{equation}{section}
  \makeatother   
  \theoremstyle{theorem} 
   \newtheorem{theorem}{Theorem}[section]
  \newtheorem{lemma}[theorem]{Lemma}
  \newtheorem{proposition}[theorem]{Proposition}
  \newtheorem{corollary}[theorem]{Corollary}
  
  \theoremstyle{definition}
  \newtheorem{remark}[theorem]{Remark} 
   \newtheorem{definition}[theorem]{Definition}
  \newtheorem{example}[theorem]{Example}

 \newcommand{\ainf}{A_\infty}
 \newcommand{\A}{\mathscr{A}}
  
 \newcommand{\C}{\mathbb{C}}
 \newcommand{\N}{\mathbb{N}}
 \newcommand{\LL}{\mathcal{L}}
 \newcommand{\PP}{\mathbb{P}}
 \newcommand{\Q}{\mathbb{Q}}

 \newcommand{\R}{\mathbb{R}}
 \newcommand{\X}{\mathbb{X}}
 
 \newcommand{\Z}{\mathbb{Z}}
 \newcommand{\io}{I}
 \newcommand{\LI}{L_\infty}
 \newcommand{\e}{\mathrm{e}}
 \newcommand{\ke}{\mathbb{K}}
 
 \newcommand{\ob}{\mathrm{Ob}}
 \newcommand{\Hom}{\mathrm{Hom}}

 \newcommand{\im}{\mathrm{Im}}
 \newcommand{\id}{\mathrm{id}}
 \newcommand{\End}{\mathop{\mathrm{End}}\nolimits}

 \newcommand{\der}{\mathrm{Der}}
 \newcommand{\sh}{\mathrm{Sh}}
 \newcommand{\ggm}{\mathrm{G}}
 \newcommand{\eul}{\mathrm{E}}
 \newcommand{\dz}{\frac{d}{dz}}
 \newcommand{\nove}{\Lambda^e}
 \newcommand{\wt}[1]{\widetilde{#1}}
 \newcommand{\no}[1]{\|#1\|}
 
 \newcommand{\ilim}[1][]{\mathop{\varprojlim}\limits_{#1}}
 \newcommand{\relmiddle}[1]{\mathrel{}\middle#1\mathrel{}}
 \newcommand{\hodot}{\widehat{\odot}}
 \newcommand{\hotimes}{\widehat{\otimes}}
 \newcommand{\hoplus}{\widehat{\oplus}}
 \newcommand{\hell}{\widehat{\ell}}
 \newcommand{\hoc}{CC^\bullet(\A)}
 \newcommand{\hoh}{CC_\bullet(\A)}
\newcommand{\rhoc}{\overline{CC}^\bullet(\A)}
 \newcommand{\rhoh}{\overline{CC}_\bullet(\A)}
 \newcommand{\vin}{\mathit{in}}
\newcommand{\vout}{\mathit{out}}
\begin{document}
  \title{Meromorphic connections in filtered $A_{\infty}$ categories}
  \dedicatory{Dedicated to Professor Kyoji Saito on his 75th birthday}
  \author{Hiroshi Ohta  
  and Fumihiko Sanda} 
  \date{\today}
 \thanks{The first named author is supported partially by JSPS Grant-in-Aid
for Scientific Research (A) No 15H02054, and the second named author is supported partially by JSPS Grant-in-Aid
for Scientific Research (A) No 15H02054 and for Young Scientists (B) No 17K17817 (PI. A. Kanazawa).}
\begin{abstract} 
In this note, introducing notions of CH module, CH morphism and CH connection, 
we define a meromorphic connection in the ``$z$-direction'' on periodic cyclic homology of an $\ainf$ category as a connection on cohomology of a CH module. Moreover, we study and clarify compatibility of our meromorphic connections under a CH module morphism preserving CH connections at chain level. Our motivation comes from symplectic geometry. 
The formulation given in this note designs to fit algebraic properties of open-closed maps in symplectic geometry.

\end{abstract}
\maketitle
   \tableofcontents
\section{Introduction}\label{Intro}
In this note we introduce and study a certain meromorphic connection on periodic cyclic homology of 
a filtered $\ainf$ category. 
In \cite{getcar} Getzler introduced the Gauss-Manin connection on periodic cyclic homology 
of an $\ainf$ algebra, which we call the  {\it Getzler-Gauss-Manin connection} in this note. 
He constructed the connection at chain level and showed that the curvature is chain homotopic to zero.
His connection does not involve derivative of the `$z$-direction'. Here the parameter $z$ 
is the auxiliary variable in the cyclic homology, which is denoted by $-u$ in \cite{getcar}.
In this article we will incorporate the derivative of the `$z$-direction' with the connection.  
This is necessary and important when we study the relationship between the Fukaya category 
for a general, not necessary Calabi-Yau, symplectic manifold and 
Kyoji Saito's flat structure \cite{saiper}. 
On the other hand, Katzarkov-Kontsevich-Pantev \cite{kkp1} introduced 
a connection on periodic cyclic homology which contains the derivative of the `$z$-direction'. 
Although the definition of our connection is inspired from their definition and  
it looks similar to theirs, they are essentially different as we will explain later
in Remark \ref{rmk:KKP}.
\par

Our motivation comes from symplectic geometry. 
We briefly and informally recall the symplectic geometric background to explain our motivation and aim of this article, though  
these geometric contents are not used in this note.
Let $(X, \omega)$ be a closed symplectic manifold. 
We do {\it not} assume $c_1(X) =0$ here.
Let 
\begin{equation}\label{eq:nov}
\Lambda_0 = \left.\left\{\sum_{i=0}^{\infty} a_i T^{\lambda_i} ~\right\vert~ a_i \in \C, \lambda_i \in \R_{\ge 0}, \lambda_i \to \infty \right\}
\end{equation}
be (a version of) the Novikov ring. 
We choose and fix a homogeneous basis $T_0=1, T_1, \dots , T_m$ of $H^{\ast}(X;\Lambda_0)$.
(For simplicity, we assume $\deg T_i$ are even.)
Let ${\bf t} = \sum_{i=0}^m t_i T_i \in H^{\ast}(X;\Lambda_0)$ and denote by $\ast_{\bf t}$ the quantum product
on $H^{\ast}(X;\Lambda_0)$  
defined by 
$$
(a\ast_{\bf t}b,c)_{PD_X} = \sum_{\alpha \in H_2(X;\Z)} \sum_{n=0}^{\infty} \frac{1}{n!}
GW_{0,\alpha, n+3}^X (\underbrace{{\bf t}, \dots , {\bf t}}_n, a,b,c)T^{\omega(\alpha)}.
$$
Here 
$(\cdot,\cdot)_{PD_X}$ denotes the Poincar\'e paring on $H^{\ast}(X;\Lambda_0)$ and 
$GW_{0,\alpha, n+3}^X$ is 
(the $\Lambda_0$ linear extension of) the genus zero $(n+3)$ points Gromov-Witten invariant of $X$ with 
class $\alpha$. 
Dubrovin introduced a meromorphic flat connection $\nabla^{\rm D}$ 
called {\it Dubrovin's quantum connection} on the trivial $H^{\ast}(X;\Lambda_0)$-bundle over 
$H^{\ast}(X;\Lambda_0)\times \PP^1$ 
satisfying 
$$
\begin{aligned}
\nabla^{\rm D}_{\frac{\partial}{\partial t_i}} & = \frac{\partial}{\partial t_i} + 
\frac{1}{z} (T_i \ast_{{\bf t}}), \\
\nabla^{\rm D}_{z\frac{\partial}{\partial z}} & = z\frac{\partial}{\partial z} -\frac{1}{z}
(E \ast_{{\bf t}}) +\mu,
\end{aligned}
$$
where $z$ is the parameter sitting in $\PP^1$, $E=c_1(X) + \sum_{i=0}^m (1-\frac{\deg T_i}{2})t_iT_i$ and $\mu \in \End H^{\ast}(X;\Lambda_0)$ defined by 
$\mu(a)= \frac{1}{2}(\deg a -\frac{\dim X}{2})a$ for $a \in H^{\ast}(X;\Lambda_0)$.
See, e.g., \cite[Lecture 3]{dubgeo} and compare Example \ref{5.10} below.
\par
On the other hand, Lagrangian intersection Floer theory on $X$ provides a filtered $A_{\infty}$ category, 
called the {\it Fukaya category}, denoted by ${\rm Fuk} (X)$. 
The {\it open-closed map} $\frak p$ introduced in \cite{fooo1} is extended in \cite{afooo} to 
a $QH^{\ast}(X)$-module homomorphism 
\begin{equation}\label{eq:ocmap}
\widehat{\frak p} ~:~ HH_{\ast}({\rm Fuk} (X);\Lambda_0) \to QH^{\ast}(X;\Lambda_0)
\end{equation}
from the Hochschild homology of the Fukaya category. Here $QH^{\ast}(X;\Lambda_0)$ denotes 
the quantum cohomology ring of $X$ with coefficients in $\Lambda_0$.
If we use the de Rham model, 
then ${\rm Fuk} (X)$ has a structure of 
cyclic filtered $A_{\infty}$ category. (See \cite{fukcyc}, \cite{fukuno}, \cite{afooo}).
The open-closed map $\widehat{\frak p}$ will be lifted to a cyclic open-closed map 
from the the cyclic homology 
of ${\rm Fuk} (X)$ to the ($S^1$-equivariant) quantum cohomology:
\begin{equation}\label{eq:cocmap}
\widehat{c\frak p} ~:~ CH_{\ast}({\rm Fuk} (X);\Lambda_0[[z]]) \to QH^{\ast}(X;\Lambda_0 [[z]]).
\end{equation}
Then general algebraic theory on the Hochschild chain/cochain of an $A_{\infty}$ category tells us 
that the Hochschild chain $CC_{\ast}({\rm Fuk} (X);\Lambda_0[[z]])$
has a structure of DGLA module over the Hochschild cochain 
$CC^{\ast}({\rm Fuk} (X);\Lambda_0[[z]])$
which has a structure of DGLA. 
Moreover it is known that $QH^{\ast}(X;\Lambda_0 [[z]])$ has a structure of hypercommutative algebra
in the sense of \cite{getope} (see Example \ref{hypcom} for the definition).
Furthermore, we note that 
these maps and categories are bulk-deformed by elements of 
$H^{\ast}(X;\Lambda_0)$. (See \cite{fooo1}, \cite{foootoric2} for more detailed discussion on bulk-deformations.)
In this situation, we like to construct a meromorphic connection on periodic cyclic homology of 
${\rm Fuk} (X)$ over $H^{\ast}(X;\Lambda_0)\times \PP^1$ which is compatible with  
Dubrovin's quantum connection $\nabla^{\rm D}$ under the cyclic open-closed map $\widehat{c\frak p}$.
This is our motivation. In this article we do not touch the geometric part of the cyclic open-closed map. 
The purpose of this article is to extract and describe the properties that the cyclic open-closed map 
is supposed to 
have in terms of purely algebraic language, and to introduce a meromorphic connection containing the derivative of the $z$-direction as well, and to prove that the connection is compatible with  
Dubrovin's quantum connection under the cyclic open-closed map in our algebraic formulation. 
In this sense 
our formulation provides an algebraic counterpart of the cyclic open-closed map in Lagrangian Floer theory
together with our meromorphic connection, 
and describes the relation of our meromorphic connection in the Fukaya $\ainf$ category to Dubrovin's quantum connection in quantum cohomology via the cyclic open-closed map. 
See Proposition \ref{prop5.21}, Corollary \ref{cor5.22} and Theorem \ref{thm:68}, 
and also Remark \ref{rem:varComp} for the variants.
The unitarity of the $A_{\infty}$ category is used in 
Section \ref{sec:CH}. We also note that there is a related result on compatibility of  
the Getzler-Gauss-Manin connection under a certain $L_{\infty}$ morphism of modules 
in deformation quantization \cite[Proposition 1.4]{catcha}.
\par
In this note a key ingredient to incorporate the derivative of the $z$-direction 
with the Getzler-Gauss-Manin connection on periodic cyclic homology of a filtered $A_{\infty}$ category
is to consider a $\Z$-grading. To encode a $\Z$-grading structure in the Fukaya $A_{\infty}$ category 
${\rm Fuk} (X)$\footnote{More precisely, we consider anchored/graded Lagrangian submanifolds as objects. See \cite{foooanch}, \cite{gradL} for more details.}
we will use the universal Novikov ring $\Lambda_0^e$ (see Example \ref{4.3} (1) for this notation) 
by adding one formal variable $e$ of degree $2$ instead of $\Lambda_0$ in \eqref{eq:nov}.
This formal variable $e$ was originally used in the universal Novikov ring in 
Lagrangian Floer theory \cite{fooo00}, \cite{fooo1} in order to encode the Maslov index of holomorphic disks, while $\Lambda_0$ or its quotient field $\Lambda$ without the formal variable $e$ 
are frequently used in recent literatures,  for example \cite{foootoric1}, \cite{foootoricmirror}, 
where only $\Z/2\Z$-grading is considered.
By using the $\Z$-grading structure we define a meromorphic connection called an {\it Euler connection}  
in Definition \ref{def:Eu},   
which incorporates the derivative of the $z$-direction with the Getzler-Gauss-Manin connection. 
\par
To establish the compatibility of our Euler connection 
with Dubrovin's quantum connection under the cyclic open-closed map 
at cohomology level, we need argument at chain level. 
In an algebraic aspect, we will study certain algebraic structure keeping track of information of chain homotopy at certain depth. 
For this purpose we introduce the notion of {\it CH module} over an $\LI$ algebra in Definition \ref{def:CH},
where `CH' stands for {\it Cartan homotopy}. It is a cousin of the notion of calculus algebra in 
\cite{dolnon}, \cite{tamnon}.
(We note that the Cartan homotopy formula already played an important role in Getzler's paper \cite{getcar} to  
construct the Getzler-Gauss-Manin connection on periodic cyclic homology.) 
Actually we will see that the structure of calculus algebra 
gives a typical example of our CH module structure and 
a hpercommutative algebra mentioned above also provides a typical example of the CH module structure.
(See Examples \ref{ex:cal}, \ref{hypcom}.) 
In the language of CH-module and CH-morphism we define,  
we formulate and study the compatibility of our Euler connections  
under a CH-morphism in Proposition \ref{prop5.21} and Corollary \ref{cor5.22}, 
and apply these results to the situation arising from an $A_{\infty}$ category in Section \ref{sec:CH}.

\begin{remark}\label{rmk:KKP}
Our connection formally looks similar to one in \cite[p.108]{kkp1} by replacing our variables $z,e$ by their variables $u,t$ respectively. In fact, the variable $z$ plays the same role as the variable $u$.
However there are indeed differences in the following points.  
Katzarkov-Kontsevich-Pantev consider a $t\in {\mathbb A}_{\C}\setminus \{0\}$-parametrized family of 
$\Z /2\Z$-graded (DG) algebras (see \cite{calcat} for its $A_{\infty}$ version), while we consider a $\Z$-graded family of (DG) algebras.
Their variable $t$ stands for a formal abstract parameter and 
does not have a non-trivial grading (but they consider a `weight' instead) because of the property $d_{{\mathcal A}_t} = t \cdot d_{A}$ etc,
described in \cite[p.108]{kkp1}. On the other hand, our variable $e$ 
appears in more primitive way and 
has a geometric meaning in Lagrangian Floer theory as mentioned above,  
which naturally involves a $\Z$-grading. 
Furthermore, since our $\Z$-graded family of (DG) algebras is different form one they consider, 
the Getzler-Gauss-Manin connections associated to these families are, at least a priori, different.  
\end{remark}

\begin{remark}\label{rmk:GPSS}
There are related works, e.g. in \cite{ganmir}, \cite{seicon} motivated also by symplectic geometry.
In \cite[Section 4.2]{ganmir} Ganatra-Perutz-Sheridan claims the compatibility of 
the Getzler-Gauss-Manin connection and Dubrovin's quantum connection under the cyclic 
open-closed map. Their Geztler-Gauss-Manin connection does not involve the derivative of 
the $z$-direction. However they consider the Calabi-Yau case. So it is actually not necessary for the  situation they study. 
Moreover, we note that a $\Z$-grading structure naturally appears in their case.
\par
In \cite{seicon} Seidel studies a similar quantum connection on equivariant Hamilton Floer homology 
and its compatibility with Dubrovin's quantum connection under the PSS map. 
\par
\end{remark}

The outline of this paper is in order.
Since we borrow Tygan's reformulation \cite{tsygau}, following Barannikov's idea sketched in \cite{barqua}, of the Getzler-Gauss-Manin connection
in terms of $L_{\infty}$ algebra and $L_{\infty}$ module, we start with recalling 
basic definitions concerning $L_{\infty}$ algebra and $L_{\infty}$ module in Section \ref{slinf}.
In Section \ref{sect:CH} we introduce the notion of CH module over an $L_{\infty}$ 
algebra $L$ and CH module morphism.
For a given $L_{\infty}$ algebra $L$ over a graded algebra $R$, we first enhance an $L_{\infty}$ algebra structure on $\widetilde{L}:=L[[z]]\oplus \epsilon L[[z]]$ with new formal variables $z$ and 
$\epsilon$ of $\deg z =2$ and $\deg \epsilon =1$, which can be regarded as a mapping cone of the morphism 
$z \cdot{\rm id} \in {\rm End }_{R[[z]]} (L[[z]])$. 
Then a {\it CH module structure} on a graded $R[[z]]$ module $\widetilde{M}$ over the $\LI$ algebra $L$ 
is defined as 
an `$\epsilon^2$-truncated' $L_{\infty}$ module structure on $\wt{M}$ over the cone $\widetilde{L}$. 
This structure includes information of chain homotopy at certain depth which will be used in later argument at chain level. On the other hand, as we mentioned above, our motivation comes from symplectic geometry, where 
the symplectic energy plays an important role. To encodes the symplectic energy
we will use the Novikov ring/field as coefficients on which the symplectic energy induces a valuation or norm. Thus in Section \ref{sec:norm} 
we define the normed (filtered) version of the notion of CH module. 
In Section \ref{sec:conn}, after defining the notion of {\it CH connection} on a CH module 
$\widetilde{M}$ over $L$, we define the Getzler-Gauss-Manin connection on 
$\widetilde{M}$ for each CH connection and Maurer-Cartan element of the $L_{\infty}$ algebra $L$ in Definition \ref{defggm}. In Subsection \ref{subeu}, using the grading operator and the Euler vector field, we incorporate the derivative of the $z$-direction with the Getzler-Gauss-Manin connection defined above. We call the resulted connection an {\it Euler connection} on $\widetilde{M}$. 
Then for any CH module morphism which intertwines CH connections we show compatibility of  
the Euler connections under the CH module morphism (Proposition \ref{prop5.21}, Corollary \ref{cor5.22}). 
Now in Section \ref{sec:CH}, we study the situation of the Hochschild chain/cochain of 
an $A_{\infty}$ category $\A$. It is known that 
the shifted (reduced) Hochschild cochain complex $\rhoc[1]$ has a structure of DGLA and 
the shifted (reduced) Hochschild chain complex $\rhoh[1]$ has a structure of DGLA module 
over $\rhoc[1]$. Moreover we show in Theorem \ref{thm:68} that 
the reduced Hochschild chain $\wt{M}:=\rhoh[1][[z]]$ has a CH module structure over  
the DGLA $L:=\rhoc[1]$. 
Therefore we can apply the story developed in up to Section \ref{sec:conn} to this situation.
Thus this derives a Getzler-Gauss-Manin connection together with derivative of the $z$-direction 
(an Euler connection) on the periodic cyclic homology of an $A_{\infty}$ category.  
Finally in Section \ref{prim},  we recall the definition of a part of primitive forms and briefly explain  relations to filtered $\ainf$ categories and Euler connections.

\par\medskip\noindent
{\bf Acknowledgment:} 
The first named author would like to thank Kyoji Saito, 
the organizer of the workshop `Primitive form' 
held at Kavli IPMU in 2014, and speakers, Si Li \cite{saipri}, Kenji Fukaya and Atsushi Takahashi 
\cite{takcal} whose talks partially encouraged him to do the research in this article.  
Since around 2008, he had several seminars with K.Saito and A.Takahashi 
on primitive forms and related things. During the seminars, 
Takahashi's comments drew his attention to importance of $\Z$-grading structure and 
the notion of calculus algebra \cite{dolfor}.

\section{Preliminaries on $\LI$ algebras and $\LI$ modules}\label{slinf}
In this section, we recall some definitions related to $\LI$ algebras (e.g. \cite{ladint}).
Let $R=\oplus_{k\in \Z} R^k$ be a $(\Z-)$graded ring.
Throughout this paper,
we assume that rings are (graded) commutative and contains $\Q$, i.e., $n \cdot 1$ is invertible for each $n \in \Z \setminus\{0\}$.
By graded $R$ modules, we mean left graded $R$ modules, which are naturally considered as graded $R$ bimodules. 
Usually, elements of graded modules are assumed to be homogeneous.

For graded modules, the degree of homogeneous elements are denoted by $|\cdot|$ and 
set $|\cdot|':=|\cdot|-1.$
For a graded $R$ module $V$, the one shift $V[1]$ is defined by $V[1]^k:=V^{k+1}$.
Let $s : V \to V[1]$ be the\\`` identity map'' (degree $-1$). 
The graded $R$ module structure of $V[1]$ is defined by $rsv:=(-1)^{|r|}srv$.
To simplify notation, we use the same letter $v$ for $sv$.

For graded $R$-modules $V, W,$ the graded tensor product over $R$ is denoted by $V \otimes W.$
The symmetric group $S_k$ of degree $k$ naturally acts on $V^{\otimes k}.$ 
The coinvariant of this action is called the graded symmetric tensor product and denoted by 
$V^{\odot k}.$
An element $\sigma \in S_k$ is called an $(i_1, i_2, \dots, i_l)$-shuffle if $i_1+\cdots+i_l=k$ and 
$$
\sigma(1) < \sigma(2) < \cdots <\sigma(i_1), \dots, \sigma(i_1+\cdots+i_{l-1}+1) < \cdots <\sigma(k).
$$
The set of $(i_1,i_2,\dots,i_l)$-shuffles is denoted by $\sh(i_1, i_2, \dots, i_l)$.
For $x_1, \dots, x_k \in V$ and $\sigma \in S_k,$ the Koszul sign $\epsilon(\sigma)$ is defined by 
\[x_1\odot \cdots \odot x_k=\epsilon(\sigma)x_{\sigma(1)} \odot \cdots \odot x_{\sigma(k)}.\] 
Note that $\epsilon(\sigma)$ depends on $x_1, \dots, x_k$.

Set $E_kV:=V[1]^{\odot k}$ and $EV:=\oplus_{0 \le k} E_kV$.
We define a product $EV \otimes EV \to EV $ and a coproduct $\Delta : EV \to EV \otimes EV$ as follows:
\begin{align*}
 &(y_1 \odot \cdots \odot y_k) \cdot (y'_1 \odot \cdots \odot y'_l)=
   y_1 \odot \cdots \odot y_k \odot y'_1 \odot \cdots \odot y'_l, \\
 &\Delta(y_1\odot \cdots \odot y_k)=\sum_{\substack{p+q=k \\ 0 \le p \le k}} 
   \sum_{\sigma \in \sh(p,q)} \epsilon(\sigma) y_{\sigma(1)}\odot \cdots \odot y_{\sigma(p)}
   \otimes y_{\sigma(p+1)}\odot \cdots \odot y_{\sigma(k)},
\end{align*}
where $\epsilon(\sigma)$ is the koszul sign determined by the shifted degree $|\cdot|'$.
We define a unit $\eta$ by the inclusion $R \cong E_0V \subset EV$ and define a counit $\epsilon$ by the projection from $EV$ to $R.$
Then $(EV, \,\cdot\,, \Delta, \eta, \epsilon)$ is a bialgebra.
We note that $\eta$ gives a coaugmentation and $(EV,\delta, \eta, \epsilon)$ is a cocomutative conilpotent cofree coalgebra (see, e.g., \cite[\S1]{lodalg}). 
For \[\{f_k\}_{1 \le k} \in \prod_{1 \le k} \Hom_R^1(V[1]^{\odot k}, V[1]),\]
we define $\widehat{f} \in \Hom_R^1(EV, EV)$ by 
\[\widehat{f}(y_1, \dots, y_k)=\sum_{\substack{p+q=k \\ 1 \le p \le k}}\sum_{\sigma \in \sh(p,q)}\epsilon(\sigma)
f_p(y_{\sigma(1)}, \dots, y_{\sigma(p)}) \odot \cdots \odot y_{\sigma(k)}.\]
This gives an isomorphism between the set $\prod_{1 \le k} \Hom_R^1(V[1]^{\odot k}, V[1])$ and the set of coderivations of $(EV, \Delta)$ such that composition of the coderivation and $\eta$ is equal to zero.  
Similarly, for \[\{f_k\}_{1 \le k} \in \prod_{1 \le k} \Hom_R^0(V[1]^{\odot k}, W[1]),\]
we define $\e^f \in \Hom_R^0(EV, EW)$ by the following:
\begin{align*}
f^{\odot l}(y_1, \dots, y_k)&=\sum_{\substack{i_1+ \cdots +i_l=k \\ \sigma \in \sh(i_1, \dots, i_l) }} 
                                                \epsilon(\sigma)
                                                f_{i_1}(y_{\sigma(1)}, \dots, y_{\sigma(i_1)})\odot \cdots \odot 
                                                f_{i_l}(y_{\sigma(i_1+\cdots+i_{l-1}+1)}, \dots, y_{\sigma(k)}),     \\
\e^f&=\sum_{l=0}^\infty \frac{f^{\odot l}}{l!}.
\end{align*}
Then this gives an isomorphism between the set $\prod_{1 \le k} \Hom_R^0(V[1]^{\odot k}, W[1])$ 
and the set of coalgebra morphisms from $EV$ to $EW$ such that the composition of the morphism and $\eta$ is equal to $\eta$ .
\begin{definition}
Let $L$ be a graded module over $R$ and 
$$
\{\ell_k\} \in \prod_{1 \le k} \Hom_R^1(L[1]^{\odot k}, L[1]).
$$
The pair $(L, \{\ell_k\})$ is called an {\it $\LI$ algebra over $R$} 
if $\{\ell_k\}$ satisfies the relation
$\widehat{\ell} \circ \widehat{\ell}=0.$
\end{definition}
\begin{remark}\label{DGLA}
Let $(L, \{\ell_k\}_{1 \le k})$ be an $\LI$ algebra. 
Set \[\delta=-\ell_1(y), \ [y_1, y_2]= (-1)^{|y_1|}\ell_2(y_1, y_2).\]
Assume that $\ell_k=0 \ (3 \le k).$
Then $(L, [\cdot, \cdot], \delta)$ is a differential graded Lie algebra $($DGLA for short$)$.
Conversely, a DGLA is naturally considered as an $\LI$ algebra.
\end{remark}
\begin{definition}
Let $(L, \{\ell_k\}), (L', \{\ell'_k\})$ be $\LI$ algebras.
A set of morphisms 
\[\{f_k\} \in \prod_{1 \le k} \Hom_R^0(L[1]^{\odot k}, L'[1])\] 
is called an {\it $\LI$ morphism} 
if it satisfies 
$\widehat{\ell'} \circ \e^f=\e^f \circ \widehat{\ell}.$
\end{definition}

We next recall the definition of $\LI$ modules over an $\LI$ algebra $(L, \{\ell_k\}).$
Let $M$ be a graded $R$-module.
We consider an $EL$ comodule $EL \otimes M[1].$
By the construction, there exists an isomorphism between the set 
$\Hom_R^1(EL \otimes M[1], M[1])$
and the set of coderivations of $EL \otimes M[1]$, 
where $(EL, \widehat{\ell})$ is considered as a differential graded coalgebra.
For an element 
\[\{\ell^M_k\} \in \prod_{0 \le k}\Hom_R^1(L[1]^{\odot k} \otimes M[1], M[1]),\]
the corresponding coderivation is denoted by 
$\widehat{\ell}^M \in \End^1_R (EL \otimes M[1])$. 
Namely, writing   
$\widehat{\ell}^M(y_1\odot \dots \odot y_k \otimes m)$ as 
$\widehat{\ell}^M(y_1,\dots,y_k |m)$, we have 
\begin{equation}\label{eq:Lmodule}
\aligned
\widehat{\ell}^M (y_1, \dots, y_k|m) 
& = \sum_{\substack{p+q=k \\ 1 \le p \le k}}\sum_{\sigma \in \sh(p,q)}\epsilon(\sigma)
\ell_p(y_{\sigma(1)}, \dots, y_{\sigma(p)}) \odot \cdots \odot y_{\sigma(k)}\otimes m \\
& + \sum_{\substack{p+q=k \\ 0 \le p \le k}}\sum_{\sigma \in \sh(p,q)}\epsilon' (\sigma) y_{\sigma(1)} \odot \dots \odot y_{\sigma(p)}\otimes
\ell^M_{q}(y_{\sigma(p+1)}, \dots, y_{\sigma(k)} | m),
\endaligned
\end{equation}
where $\epsilon'(\sigma)\!=\!(-1)^{|y_{\sigma(1)}|'+ \cdots +|y_{\sigma(p)}|'}\epsilon(\sigma)$.
Note that $\widehat{\ell}^M$ depends on the $\LI$ algebra structure of $L$.
\begin{definition}
$(M, \{\ell^M_k\}_{0 \le k})$ is called an {\it $\LI$ module over $L$} if 
$\widehat{\ell}^M \circ \widehat{\ell}^M=0.$ 
\end{definition} 
\begin{remark}\label{DGLAM}
Let $L$ be a DGLA and $(M, d)$ be a DGLA module over the DGLA $L$.
The action of $y \in L$ is denoted by $\LL_y \in \End^{|y|}_R(M).$
We consider $L$ as an $\LI$ algebra $($see Remark \ref{DGLA}$)$.
Set \[\ell_0^M=-d, \quad \ell_1^M(y|m)=(-1)^{|y|}\LL_y(m), \quad \ell_k^M=0 \ (k \ge 2).\]
Then $(M, \{\ell^M_k\}_{0 \le k})$ is an $\LI$ module. 
In this way, a DGLA module $(M,d,\LL)$ over a DGLA $L$
can be regarded as an $L_{\infty}$ module.
\end{remark}
We recall the definition of morphisms of $\LI$ modules.
Let $(N, \{\ell_k^N\}_{0 \le k})$ be another $\LI$ module over $L.$
Then there exists an isomorphism between the set $\Hom_R^0 (EL \otimes M[1], N[1])$ and 
the set of comodule morphisms from $EL \otimes M[1]$ to $EL \otimes N[1].$
For an element
\[\{f_k\}_{0 \le k} \in \prod_{0 \le k} \Hom^{0}_R (L[1]^{\odot k} \otimes M[1], N[1]),\]
the corresponding comodule morphism is denoted by $\check{f}.$
\begin{definition}
$\{f_k\}_{0 \le k}$ is called an {\it $\LI$ module morphism} if it satisfies
$\widehat{\ell}^N \circ \check{f}=\check{f} \circ \widehat{\ell}^M.$  
\end{definition}
\begin{remark}\label{rel}
Let $L$ and $L'$ be $\LI$ algebras over $R$, 
$\{f_k\}_{1 \le k}$ be an $\LI$ morphism from $L$ to $L',$ and 
$(M, \{\ell_k^M\}_{0 \le k})$ be an $\LI$ module over $L'$. 
Then $(\e^{-f} \otimes \id) \circ \widehat{\ell}^M \circ (\e^f \otimes \id)$ is a coderivation and 
this coderivation makes $M$ into an $\LI$ module over $L$. 
Using this construction, we can define $\LI$ module morphisms between $\LI$ modules defined over different $\LI$ algebras.
Similar remarks are applied to the cases of morphisms of 
CH-modules (Definition \ref{def:CHhom}) and normed CH-modules (Definition \ref{def:CHmor}).
\end{remark}

\section{CH structures}\label{sect:CH}

\subsection{Notations on formal power series}\label{s3.1}
Let $M$ be a graded additive group and 
let $t_1, t_2, \dots, t_k$ be formal variables with degree $d_1, d_2, \dots, d_k \in \Z$ respectively.
The space of degree $d$ formal power series of 
$t_1, \dots , t_k$ (these variables are graded commutative) with coefficients in $M$ 
is denoted by $M[[t]]^d$ and set $M[[t]]:=\oplus_{d \in \Z} M[[t]]^d.$
Note that the degrees of coefficients also contribute to the degrees of elements of $M[[t]].$
The space of formal Laurent power series $M((t))=\oplus_{d \in \Z} M((t))^d$ is defined similarly.
These are naturally considered as graded additive groups.

A morphism from $M$ to another graded additive group $N$ naturally extends to a morphism
from $M[[t]]$ (resp. $M((t))$) to $N[[t]]$ (resp. $N((t))$).
By abuse of notation, these extended morphisms are also denoted by the same symbols.  

We use multi-index notation, i.e.,
$t^\alpha:=t_1^{\alpha_1} \cdots t_k^{\alpha_k}$ for $\alpha=(\alpha_1, \dots, \alpha_k) \in \Z^k.$
If $R$ is a graded ring and $M$ is a graded module over $R$, then $R[[t]]$ (resp. $R((t))$) has a natural graded ring structure and $M[[t]]$ (resp. $M((t))$) has a natural (left) graded module structure over $R[[t]]$ (resp. $R((t))$).
Note that the ring structure and module structure are determined by the following sign rules:
\[t^\alpha r=(-1)^{|r||t^\alpha|}rt^\alpha,\  
t^\alpha m=(-1)^{|m||t^\alpha|}mt^\alpha, \ \ 
(r \in R, m \in M).\]

\subsection{CH structures}
In \S 3.2, we introduce CH structures. 
This construction is inspired by \cite{tsygau}.
Let $(L, \{\ell_k\}_{1\le k})$ be an $\LI$ algebra over a graded algebra $R$.
Let $z$ and $\epsilon$ be formal variables with degree 
$$
\vert z \vert =2, \quad \vert \epsilon \vert =1.
$$
We introduce a ``mapping cone'' of the morphism $z \cdot\id \in \End(L[[z]]).$
Set 
$$
\wt{L}:=L[[z]] \oplus \epsilon L[[z]].
$$
Then $\wt{L}[1]$ is naturally identified with $L[1][[z]]\oplus \epsilon L[1][[z]].$
Note that $\epsilon \circ s=-s \circ\epsilon$, where $\epsilon$ is the multiplication by $\epsilon$.
We define operations $\wt{\ell}_k$ on $\wt{L}[1]$ by the following equations:
\begin{equation}\label{eq:tildl}
\wt{\ell}_k(y_1+\epsilon y_1', \dots, y_k+\epsilon y_k'):=
   \ell_k(y_1, \dots, y_k)-\epsilon\sum_{i=1}^k(-1)^\# \ell_k(y_1, \dots, y_i', \dots, y_k)+
   \begin{cases}zy_1' &\text{if} \ k=1,\\
                      0     &\text {if} \ k\ge 2,
    \end{cases}
\end{equation}
where $\#:=|y_1|'+\cdots+|y_{i-1}|'$ and $|y_i|' = |y_i| + 1$ for $y_i \in L[[z]]$.
\begin{proposition}
$(\wt{L}, \{\wt{\ell}_k\}_{1 \le k})$ is an $\LI$ algebra over $R[[z]]$.
\end{proposition}
\begin{proof}
We put
 \[\wt{\ell}'_k(y_1+\epsilon y_1', \dots, y_k+\epsilon y_k'):=
   \ell_k(y_1, \dots, y_k)-\epsilon\sum_{i=1}^k(-1)^\# \ell_k(y_1, \dots, y_i', \dots, y_k). \]
Then it is easy to see that they satisfy the $\LI$ relations. 
Hence it is sufficient to show that the coefficient of $z$ in the formula 
 \[\wt{\ell}_1 \wt{\ell}_k(y_1+\epsilon y_1', \dots, y_k+\epsilon y_k')+
    \sum_{i=1}^k (-1)^\# \wt{\ell}_k(y_1+\epsilon y_1', \dots, \wt{\ell}_1(y_i+\epsilon y_i'),\dots, y_k+\epsilon y_k')\]
    is equal to zero, where $\#:=|y_1+\epsilon y_1'|'+\cdots+|y_{i-1}+\epsilon y_{i-1}'|'$.
This statement follows by direct calculation. 
\end{proof}
Let $(L', \{\ell'_k\}_{1 \le k})$ be another $\LI$ algebra. 
For an element 
$$
\{f_k\} \in \prod_{1 \le k} \Hom_R^{0}(L[1]^{\odot k}, L'[1]),
$$ we set 
\begin{equation}\label{eq:wtf}
\wt{f}_k(y_1+\epsilon y_1', \dots, y_k+\epsilon y_k'):=
   f_k(y_1, \dots, y_k)+\epsilon \sum_{i=1}^k(-1)^\# f_k(y_1, \dots, y_i', \dots, y_k),
\end{equation}
where $\#:=|y_1+\epsilon y_1'|'+\cdots+|y_{i-1}+\epsilon y_{i-1}'|'$.
\begin{proposition}\label{prop:wtf}
If $\{f_k\}_{1 \le k}$ is an $\LI$ morphism from $L$ to $L'$, then $\{\wt{f}_k\}_{1 \le k}$ is also an $\LI$ morphism from $\wt{L}$ to $\wt{L'}$.
\end{proposition}
\begin{proof}
It is sufficient to show that the coefficient of $z$ in the formula
\[\wt{\ell}_1 \wt{f}_k(y_1+\epsilon y_1', \dots, y_k+\epsilon y_k')-
    \sum_{i=1}^k (-1)^\# \wt{f}_k(y_1+\epsilon y_1', \dots, \wt{\ell}_1(y_i+\epsilon y_i'), 
    \dots, y_k+\epsilon y_k')\]
is equal to zero, where $\#:=|y_1+\epsilon y_1'|'+\cdots+|y_{i-1}+\epsilon y_{i-1}'|'.$
This follows by direct calculation.
\end{proof}
Now we consider the following decreasing filtration on $\wt{L}[1]$ which is defined by the number of 
$\epsilon$:
 \[\cdots F^{-2}:=\wt{L}[1] \supseteq 
 F^{-1}:=\wt{L}[1] \supset F^0:=L[1][[z]] \supset F^1:=0 \supseteq \cdots
 .\]
Let $\wt{M}$ be a graded $R[[z]]$ module. 
Then $\wt{M}[1]$ is equipped with the following filtration:
      \[ \cdots F^{-1}:=\wt{M}[1] \supseteq 
      F^0:=\wt{M}[1]\supset F^1:=0 \supseteq \cdots
      .\]
We note that tensor products of modules with filtrations are naturally equipped with filtrations.
We consider a set of morphisms 
\[\ell_k^{\wt{M}} \in \Hom_{R[[z]]}^{1} 
  (\wt{L}[1]^{\odot k}\otimes\wt{M}[1], \wt{M}[1]) \ \ (0 \le k).\]
Note that $\odot$ and $\otimes$ are defined over $R[[z]]$. 
We denote by 
$$
\widehat{\ell}^{\wt{M}} \in {\rm Hom}_{R[[z]]}^1  
(E\wt{L}\otimes\wt{M}[1], E\wt{L}\otimes\wt{M}[1])
$$ 
the coderivation corresponding to $\ell_k^{\wt{M}}$.   
\begin{definition}
Let $(L, \{\ell_k\}_{1\le k})$ be an $\LI$ algebra over a graded algebra $R$ and 
$\wt{M}$ a graded $R[[z]]$ module.  
Let $n \in \Z_{\ge 0}$. A set of morphisms $\{\ell_k^{\wt{M}}\}_{0 \le k}$ is called an 
{\it $\LI$ module structure on $\wt{M}$ over $\wt{L} \ \mathrm{mod} \ \epsilon^n$}  if it satisfies  
\begin{equation}\label{eq:mod}
(\widehat{\ell}^{\wt{M}}\circ\widehat{\ell}^{\wt{M}})(F^a) \subset F^{a+n}
\end{equation} 
for any $a \in \Z$.
Here $F^a$ is the filtration on $E\wt{L} \otimes \wt{M}[1]$ induced by the filtrations on 
$\wt{L}[1]$ and $\wt{M}[1]$ defined as above. 
\end{definition}

\begin{remark}\label{3.4}
Let $\epsilon : E\wt{L} \to R[[z]]$ be the counit.
By the explicit formula (\ref{eq:Lmodule}) 
(modify the sign $\epsilon'(\sigma)$ to $\epsilon(\sigma)$), 
the morphism $\epsilon \otimes \id$ gives an isomorphism from
the set of coderivations (of degree $2$) of $E\wt{L} \otimes \wt{M}[1]$ to
$\Hom^2_R(E\wt{L} \otimes \wt{M}[1], \wt{M}[1])$ and this isomorphism preserves the filtrations.
Moreover 
\[F^{a+n}\wt{M}[1]=\begin{cases} \wt{M}[1]  \ &(a+n \le 0),\\
                                              0 \ &(a+n \ge 1). \end{cases}\]
Hence the condition $(\widehat{\ell}^{\wt{M}}\circ\widehat{\ell}^{\wt{M}})(F^a) \subset F^{a+n}$ is equivalent to
the condition 
\begin{align*}
\big{(}(\epsilon \otimes \id) \circ \widehat{\ell}^{\wt{M}}\circ\widehat{\ell}^{\wt{M}}\big{)}(F^{1-n})=0.
\end{align*}

\end{remark}

\begin{remark}\label{3.5}
When $n=1$, the condition \eqref{eq:mod} yields 
$\ell^{\wt{M}} \circ \widehat{\ell}^{\wt{M}} (F^0) =0$. 
Thus an $\LI$ module structure on $\wt{M}$ over $\wt{L}$ $\mod \epsilon$ is nothing but an $\LI$ module structure on $\wt{M}$ over $L[[z]]$. When $n=2$, 
the condition \eqref{eq:mod} implies 
$\ell^{\wt{M}} \circ \widehat{\ell}^{\wt{M}} (F^{-1})=0.$ 
\end{remark}

\begin{definition}\label{def:CH}
An $\LI$ module structure on $\wt{M}$ over $\wt{L}$ $\ \mathrm{mod} \ \epsilon^2$ 
is called a {\it CH structure on $\wt{M}$ over $L$}. 
A graded $R[[z]]$ module $\wt{M}$ with a CH structure over $L$ is called a {\it CH module over $L$}.
\end{definition}


Let $(\wt{M}, \{\ell^{\wt{M}}_k\}_{0 \le k})$ be a CH module over an $\LI$ algebra $(L, \{\ell_k\}_{0 \le k})$. 
We put 
\begin{equation}\label{ch3.2.1}
\delta:=-\ell_1, \quad [y_1, y_2]:=(-1)^{|y_1|}\ell_2(y_1, y_2) 
 \ \text{for all} \ y_1, y_2 \in L,
\end{equation}
and   
\begin{equation}\label{ch3.2.2}
d:=-\ell_0^{\wt{M}}, \quad \LL_y:=(-1)^{|y|}\ell_1^{\wt{M}}(y|\,\cdot\,)
  \ \text{for all} \ y \in L[[z]] \subset \wt{L}.
\end{equation}
Moreover, associated to the CH module structure,  
we define maps $\io_{y_1}$ and $\rho_{y_1, y_2}$ for $y_1, y_2 \in L[[z]]$ by  
\begin{equation}\label{ch3.3}
 \io_{y_1}:=(-1)^{|y_1|+1}\ell_1^{\wt{M}}(\epsilon y_1|\, \cdot \,), \ \quad 
 \rho_{y_1, y_2}:=(-1)^{|y_1|+|y_2|}\ell^{\wt{M}}_2(y_1, \epsilon y_2|\,\cdot\,).
\end{equation}

Now we write down explicitly the relations of the CH module structure under the following conditions:
\begin{equation}\label{ch3.1} 
\begin{cases}
& \ell_k=0 \ (3 \le k),\\
& \ell_k^{\wt{M}} =0 \ (3 \le k), 
\ \  \ell_2^{\wt{M}}(y_1, y_2| \ \cdot\ )=0 
\ \text{for all} \ y_1, y_2 \in L[[z]] \subset \wt{L}.  
\end{cases}
\end{equation}
Note that $(L, \delta, [\cdot, \cdot])$ is a DGLA over $R$ and $(\wt{M},d, \LL)$ is a DGLA module over the DGLA $L[[z]],$ where the DGLA structure on $L$ linearly extends to $L[[z]]$.  
\begin{remark}\label{lsign}
We will consider $\LL$ $($resp. $\io)$ as an element of 
$$
\Hom_{R[[z]]}\big{(}L[[z]], \End_{R[[z]]}(\wt{M})\big{)}
$$
of degree $0$ $($resp. degree $1)$.
Similarly, we will consider 
$$
\rho \in 
\Hom^0_{R[[z]]}\big{(}L[[z]] \otimes L[[z]], \End_{R[[z]]}(\wt{M})\big{)}.
$$
\end{remark}   
The $\LI$ relations$\mod \epsilon ^2$ yield the following relations 
among the operators (see Remark \ref{3.4}):
\begin{align} \label{ch3.4}
&[d, \io_y]+\io_{\delta y}+z\LL_y=0, \\ \label{ch3.5}
&\io_{[y_1, y_2]}-(-1)^{|y_1|}[\LL_{y_1}, \io_{y_2}]+
        [d, \rho_{y_1, y_2}]-\rho_{\delta y_1, y_2}-(-1)^{|y_1|}\rho_{y_1, \delta y_2}=0, \\ \label{ch3.6}
&\rho_{[y_1, y_2], y_3}-\rho_{y_1, [y_2, y_3]}+(-1)^{|y_1||y_2|}\rho_{y_2, [y_1, y_3]}
-[\LL_{y_1}, \rho_{y_2, y_3}]+(-1)^{|y_1||y_2|}[\LL_{y_2}, \rho_{y_1, y_3}]=0.
\end{align}
Conversely, we easily see the following:
\begin{proposition}\label{chprop3.7}
Let $(L, \delta, [\,\cdot, \cdot\,])$ be a DGLA over $R$ and $(\wt{M}, d, \LL)$ be a DGLA module over 
the DGLA $L[[z]]$ equipped with morphisms $\io$ and $\rho$ 
where 
$$
\aligned
\io & \in \Hom^1_{R[[z]]}\big{(}L[[z]], \End_{R[[z]]}(\wt{M})\big{)}, \\
\rho & \in \Hom^0_{R[[z]]}\big{(}L[[z]] \otimes L[[z]], \End_{R[[z]]}(\wt{M})\big{)}.
\endaligned
$$
Suppose that these morphisms satisfy the relations \eqref{ch3.4}, \eqref{ch3.5} and 
\eqref{ch3.6}.
Then, by the formulas \eqref{ch3.2.1}, \eqref{ch3.2.2} and \eqref{ch3.3}, these morphisms give a CH structure on $\wt{M}$ over $L$ which satisfies the conditions \eqref{ch3.1}. 
\end{proposition}
\begin{example}\label{ex:cal}
({\bf Calculus algebra}.) To give an example of CH module, 
we recall the definition of calculus $($e.g., \cite{dolnon}, \cite{tamnon}$)$. 
Let $(V, \wedge)$ be a graded commutative algebra over $R$, $W$ be a graded module over $V$
with the structure morphism $\iota_\bullet : V\otimes W \to W$ and $B \in \End^{-1}_R(W)$ be a degree $-1$ morphism with $B^2=0.$ 
Assume $V[1]$ is equipped with a graded Lie algebra structure $[\cdot, \cdot]$ and $W[1]$ is equipped with a graded Lie module structure 
$\LL_\bullet : V[1] \otimes W[1] \to W[1]$.
For $x \in V$, the morphisms $(-1)^{|x|'}[x, \cdot\,] \in \End_R(V)$ and
$(-1)^{|x|'}\LL_x \in \End_R(W)$ are denoted by $l_x$.
The 7-tuple $(V, W, \wedge, [\cdot, \cdot], \iota_\bullet, \LL_\bullet, B)$ is called a 
{\it calculus}
if they satisfy 
 \begin{align*}
  l_{x_1}(x_2 \wedge x_3)=(l_{x_1}x_2) \wedge x_3 + (-1)^{(|x_1|+1)|x_2|}x_2 \wedge l_{x_1}x_3
 \end{align*}
and
 \begin{align*}   
  l_{x_1 \wedge x_2}=l_{x_1} \circ \iota_{x_2}+(-1)^{|x_1|}\iota_{x_1} \circ l_{x_2},\ \ \ 
  \iota_{l_{x_1}x_2}=[l_{x_1}, \iota_{x_2}],\ \ \ 
  l_{x}&=[B, \iota_x].
 \end{align*}
For a calculus, set 
\[ L:=V[1], \ \wt{M}:=W[1][[z]], \ d:=zB, \ \LL_{x}:=\LL_{x}, \ \io_{x}:=(-1)^{|x|} \iota_x \ (x \in V).\]  
Then $\wt{M}$ is a CH module over $L$ with 
$\ell^{\wt{M}}_k=0, \ (k \ge 2)$. 
\end{example}
\begin{example}\label{hypcom}
({\bf Hypercommutative algebra}.) 
We recall the definition of hypercommutative algebra \cite{getope}.
Let $A$ be a graded $R$ module equipped with a symmetric $k$-ary operation 
$(\,\cdot, \dots, \cdot\,) : A^{\odot k} \to A$ of degree $4-2k$ for each $k \ge 2.$
For a subset $S=\{i_1, i_2, \dots, i_k\} \subset \N  \ (i_1 < \cdots < i_k)$ and 
$x_{i_1}, \dots, x_{i_k} \in A$, we will denote $x_{i_1} \odot \cdots \odot x_{i_k}$ by $x_S$. 
We call $A$ a {\it hypercommutative algebra} if 
\[\sum_{S_1 \sqcup S_2=\{3,4, \dots, k-1\}}\pm(x_1, x_{S_1}, (x_2, x_{S_2}, x_k))= 
   \sum_{S_1 \sqcup S_2=\{3,4, \dots, k-1\}}\pm(x_2, x_{S_1}, (x_1, x_{S_2}, x_k)),\]
where $\pm$ are the Koszul signs. 
These equations are called the {\it WDVV equations}.
\par
Set $L:=A[1]$, which is considered as an abelian graded Lie algebra.
We also set $\wt{M}:=A[[z]].$
We define $\ell^{\wt{M}}_k$ by
 \[\ell^{\wt{M}}_k(x_1, \dots, x_k|m):=0, \ 
  \ell^{\wt{M}}_k(\epsilon x_1, x_2, \dots, x_k|m):=
  (-1)^{1+|x_1|+|x_2|+ \cdots+|x_k|} (x_1, x_2, \dots, x_k, m),\]
where $x_1, \dots, x_k, m \in A[[z]]$. 
Note that  the double suspensions $s^2x_i$ are also denoted by the same symbol $x_i.$
Then $\wt{M}$ is a CH module over $L$.
Moreover, if we define $\ell^{\wt{M}}_k(\epsilon x_1, \epsilon x_2, \dots, x_k|m)=0$, we obtain a trivially extended $L_{\infty}$ module 
mod $\epsilon^3$ structure in this example by using the $WDVV$ equations.
\end{example}

We next define morphisms of CH modules.
Let $\wt{M}$ and $\wt{N}$ be CH modules over $L$.
We consider a set of morphisms 
$$
\{f_k\}_{0 \le k} \in \prod_{0 \le k} \Hom^{0}_{R[[z]]}(\wt{L}[1]^{\odot k}\otimes\wt{M}[1], \wt{N}[1]).
$$ 
\begin{definition}\label{def:CHhom}
The set $\{f_k\}_{0 \le k}$ is called a 
{\it CH module morphism from $\wt{M}$ to $\wt{N}$} if it satisfies  
\[(\widehat{\ell}^{\wt{N}} \circ \check{f} - \check{f} \circ \widehat{\ell}^{\wt{M}})(F^a)
    \subset F^{a+2}\] 
for all $a \in \Z.$
\end{definition}
\begin{remark}\label{relch}
Let $L, L'$ be $\LI$ algebras and $\{f_k\}_{1 \le k}$ be an $\LI$ morphism from $L$ to $L'$.
Then the morphism $\e^{\wt{f}}$ preserves the filtrations. 
Hence a CH module over $L'$ naturally gives a CH module over $L$ $($see also Remark \ref{rel}$)$.   
\end{remark}
Let $\{f_k\}_{0 \le k}$ be a CH module morphism from $\wt{M}$ to $\wt{N}$.  
For $y \in L[[z]]$ and $m \in \widetilde{M}$ we put 
\begin{equation}\label{chmor3.7}
F_y(m):=(-1)^{|y|}f_1(y|m), \quad  
F_y^{\epsilon}(m):=(-1)^{|y|+1}f_1(\epsilon y|m).
\end{equation}
Thus we have $F, F^{\epsilon} \in \Hom_{R[[z]]} \left(L[[z]],\End_{R[[z]]}(\wt{M},\wt{N})\right)$.
Then we find 
\begin{equation}\label{chmo}
\io_y \circ f_0-f_0 \circ \io_y=d \circ F_y^{\epsilon}-(-1)^{|y|}F_y^{\epsilon} 
\circ d-F_{\delta y}^{\epsilon}-zF_y.
\end{equation}
Here $f_0 \in \End_{R[[z]]}(\wt{M}[1],\wt{N}[1])$ is naturally considered as an element of $\End_{R[[z]]}(\wt{M},\wt{N})$.

\section{Normed objects}\label{sec:norm}
\subsection{Preliminaries on norms}
A main reference of this subsection is \cite{bgr}.
We basically follow the terminology used there.
Let $V=\oplus_{k \in \Z} V^k$ be an graded additive group. 
A (non-archimedean) norm on $V$ is a function $\no{\cdot} : V \to \R_{\ge 0}$ with the following properties:
\begin{itemize}
\item $\no{v_1-v_2} \le \max \{\no{v_1}, \no{v_2}\}$ for all $v_1, v_2 \in V$.
\item $\no{v}=0$ if and only if $v=0$. 
\item $\no{v}=\displaystyle{\max_{k \in \Z}} \no{v_k}$, where $v=\sum v_k  \ (v_k \in V^k).$
\end{itemize}
A graded additive group with a norm is called a {\it graded normed group}.
A graded normed group $V$ is said to be {\it complete} if each $V^k$ is a complete metric space
with respect to the norm.
Set $\no{V}:= \sup\{\no{v}\ |\ v \in V\}$. 
A graded normed group $V$ is said to be {\it bounded} 
if $\no{V} \le C$ for some constant $C \in \R_{\ge 0}$.

Let $W$ be another graded normed group.
Then $V \otimes W$ is equipped with a norm which is defined by
$\|\sum_{i=1}^kv_i \otimes w_i\|=\sup_{1 \le i \le k}\no{v_i} \cdot \no{w_i}.$
The completion of $(V\otimes W)^k$ is denoted by $(V \hotimes W)^k.$
The completed tensor product is defined by $V\hotimes W :=\oplus (V \hotimes W)^k$. 
The completed symmetric tensor product $\widehat{\odot}$ and
the completed direct sum $\hoplus$ are defined similarly. 
 
A group morphism $f : V \to W$ is said to be {\it contractive}  
if $\no{f(v)} \le \no{v}$ for all $v \in V.$

Let $R$ be a graded ring and $\no{\cdot}$ be a norm on $R$ (as a graded additive group).
The norm $\no{\cdot}$ is called a {\it ring norm} if it satisfies 
$\no{r_1 r_2} \le \no{r_1}\no{r_2} \ (r_1, r_2 \in R)$ and $\ \no{1}=1.$
A graded ring equipped with a ring norm is called a {\it graded normed ring}.

Let $V$ be a graded module over a graded normed ring $R$ and $\no{\cdot}$ be a norm on $V$.
The norm $\no{\cdot}$ is called a {\it module norm} if it satisfies 
$\no{rv} \le \no{r}\no{v} \ (r \in R, v \in V).$
A graded module equipped with a module norm is called a {\it graded normed module}.

Let $\ke$ be a normed ring (ungraded, i.e., concentrated in degree zero).
Let $R$ be a graded normed algebra over $\ke$, i.e.,
$R$ is a graded algebra over $\ke$ equipped with a norm $\no{\cdot}$ such that 
$\no{\cdot}$ is a ring norm and a $\ke$ module norm.

\begin{example}\label{ex:normedring}
Here are some examples of normed rings $\ke$.
\begin{enumerate}
\item A field with the trivial valuation.
\item The universal Novikov field 
  \[\Lambda:=\displaystyle{\left\{\sum_{i=0}^\infty a_i T^{\lambda_i}\relmiddle|
     \begin{aligned} &a_i \in \C, \lambda_i \in \R, \lim_{i \to \infty}\lambda_i=\infty \\
                          &\lambda_0 < \lambda_1 <\lambda_2 < \cdots
     \end{aligned}     \right\}}.\]
        The norm on $\Lambda$ is defined by 
        $\|\displaystyle{\sum_{i=0}^\infty a_i T^{\lambda_i}\|:=\e^{-\lambda_0}}.$
        Note that $\Lambda$ is a non-archimedean valuation field. 
\item The Novikov ring 
$\Lambda_0 = \{ x \in \Lambda ~\vert~ \Vert x \Vert \le 1 \} \ ($the valuation ring of $\Lambda),$ 
         where the norm is induced from  $\Lambda$.   
\end{enumerate}
\end{example}
\begin{example}\label{4.3}\
Here are some examples of graded normed algebras and graded normed modules.
\begin{enumerate}
\item Let $R$ be an $($ungraded$)$ normed algebra over $\ke$ $($e.g., $R=\Lambda_0, \ke=\C)$.
         Set $R^e:=R((e)),$ where $e$ is a formal variable of degree $2$.
         Note that $R((e)) =\oplus_k R((e))^k =R[e, e^{-1}]$ (see \S \ref{s3.1}).
         Set \[\displaystyle{\|\sum_{k\in \Z} r_k e^k\|:=\max_{k \in \Z}\|r_k\|}, \ \ (r_k \in R).\]
         Then $R^e$ is a graded normed algebra over $\ke$. 
\item Let $R$ be a graded normed algebra over $\ke$ and 
         $z$ be a formal variable of degree $2$.
         Assume that $R$ is bounded.
         We define a norm on $R[[z]]$ by 
         \[\no{\sum_{k=0}^\infty r_kz^k}:=\sup_{k \in \Z_{\ge 0}}\no{r_k}.\]
         Then $R[[z]]$ is a graded normed algebra.
         
         Let $M$ be a graded bounded normed module over $R.$
         Similar to $R[[z]]$, $M[[z]]$ is also equipped a norm and 
         $M[[z]]$ is a graded normed module over $R[[z]].$   
\item Let $R$ be a graded normed algebra over $\ke$ 
         and let $t^1, \dots, t^m$ be formal variables with degree $d_1, \dots ,d_m \in \Z$.
         Assume that $R$ is bounded.
         Take some constant $C \in \R_{>1}.$
          We define a norm on $R[[t]]$ by 
          \[\no{\sum_\alpha r_\alpha t^\alpha}:=
          \sup_\alpha \no{r_\alpha}\no{t^\alpha},\]
          where $\no{t^\alpha}:=C^{-\alpha_1-\cdots-\alpha_m}.$
          Then $R[[t]]$ is a graded normed algebra over $\ke$.
          
          Let $M$ be a graded bounded normed module over $R.$
          Similar to $R[[t]]$, $M[[t]]$ is also equipped with a norm and 
          $M[[t]]$ is a graded normed module over $R[[t]].$ 
\end{enumerate}
\end{example}
\begin{remark}
If a graded bounded normed group $V$ is complete, then $V[[z]], V[[t]]$ are also complete.
\end{remark}

\subsection{Normed $\LI$ algebras and modules}

Let $R$ be a graded normed algebra over a normed ring $\ke$.
Assume that $R$ is complete.
Let $L$ be an $\LI$ algebra equipped with a complete $R$ module norm.  
Set $\widehat{E}L:=\hoplus L[1]^{\hodot k}.$
As in the case of $EL$, we can define morphisms $\cdot, \Delta, \eta, \epsilon$ on $\widehat{E}L.$
Morphisms $\ell_k \in \Hom^1_R(L[1]^{\odot k}, L[1]) \ (1 \le k)$ with 
$\no{\ell_k(y_1, \dots, y _k)} \le \no{y_1} \cdots \no{y_k}$
naturally extend to a contractive coderivation $\widehat{\ell}$ of $\widehat{E}L$ with $\widehat{\ell} \circ \eta=0$.  
Also as in the case of $EL,$ this correspondence gives an isomorphism.
A set of morphisms $\{\ell_k\}_{1 \le k}$ corresponding to a contractive morphism $\widehat{\ell}$ is also said to be {\it contractive}. 
Similarly, we can define contractive morphisms $\e^f, \widehat{\ell}^M, \check{f}$ (see \S \ref{slinf}) and corresponding sets of morphisms are also said to be contractive.
\begin{definition}
A {\it normed $\LI$ algebra} is a pair $(L, \{\ell_k\}_{1\le k})$,where
\begin{itemize}
\item $L$ is a graded complete normed $R$-module.
\item $\{\ell_k\}$ is an $\LI$-structure on $L$.
\item $\{\ell_k\}$ is contractive.
\end{itemize}
\end{definition}
\begin{definition}
A morphism between normed $\LI$ algebras is defined by a contractive $\LI$ morphism, which is called a {\it normed $\LI$ morphism}.
\end{definition}

\begin{definition}
A {\it normed $\LI$ module} over a normed $\LI$ algebra $L$ is a pair $(M, \{\ell^M_k\}_{0 \le k})$, where
\begin{itemize}
\item $M$ is a graded complete normed $R$-module.
\item $\{\ell^M_k\}$ is an $\LI$-module structure on $M$ over $(L, \{\ell_k\}_{1\le k}).$
\item $\{\ell^M_k\}$ is contractive.
\end{itemize}
\end{definition}
\begin{definition}
A morphism between normed $\LI$ modules is defined by a contractive $\LI$ module morphism,
which is called a {\it normed $\LI$ module morphism}.
\end{definition}

\subsection{Normed CH modules}
We assume $\no{R} \le 1.$
Let $L$ be a normed $\LI$ algebra with $\no{L} \le 1.$
We define a norm on $\wt{L}$ by $\no{y+\epsilon y'}:=\max\{\no{y}, \no{y'}\}$
(see Example \ref{4.3} (2) for the definition of the norm on $L[[z]]$). 
Then $(\wt{L}, \{\wt{\ell}_k\})$ is a normed $\LI$ algebra.
\begin{definition}
A {\it normed CH module} over a normed $\LI$ algebra 
 $L$ is a pair $(\wt{M}, \{\ell^{\wt{M}}_k\}_{0 \le k})$, where
\begin{itemize}
\item $\wt{M}$ is a graded complete normed $R[[z]]$ module.
\item $\{\ell^{\wt{M}}_k\}$ is a CH structure on $\wt{M}$.
\item $\{\ell^{\wt{M}}_k\}$ is contractive.
\end{itemize}         
\end{definition}
\begin{definition}\label{def:CHmor}
A morphism between normed CH modules is defined by a contractive CH module morphism,
which is called a {\it normed CH morphism}. 
\end{definition}
%


\section{Getzler-Gauss-Manin connections}\label{sec:conn}

\subsection{Connections on CH modules}\label{5.1}
Let $R$ be a graded algebra over a ring $\ke.$
In this subsection $R$ and $\ke$ are not assumed to be normed.
The graded Lie algebra of derivations of $R$ is denoted by $\der_{\ke}(R)$, i.e., 
\[\der_{\ke}(R)=\oplus_k \der_{\ke}^k(R), \ \der_{\ke}^k(R):=\left\{X \in \End_{\ke}^k (R) \relmiddle| X(rr')=X(r)r'+(-1)^{k|r|}rX(r')\right\}.\]
\begin{definition}
We define $E \in \der_\ke^0(R)$ by $E(r):=\frac{1}{2}|r|r$.
This vector field is called an {\it Euler vector field}.
\end{definition}
For a graded $R$ module $V$, a {\it connection on $V$} is a (degree preserving) morphism 
\[\nabla : \der_{\ke}(R) \to \End_\ke(V)\]
such that $\nabla_X(rv)=(Xr)v+(-1)^{|X||r|}r\nabla_Xv$, 
for $X \in \der_{\ke}(R), r \in R, v \in V$. 
For a connection $\nabla$ on $V$, a connection on $V[1]$ is defined by the formula 
$(-1)^{|X|}\nabla_X.$
For another graded $R$ module $W$ with a connection $\nabla$, graded $R$ modules
$V \oplus W, V \otimes W, \Hom_R(V, W)$ are also equipped with  connections.
These connections  are also denoted by $\nabla.$
Explicitly, the connections on $V \otimes W$ and $\Hom_R(V, W)$ are defined by the following formulas:
\begin{align*}
\nabla_X(v \otimes w)&=(\nabla_Xv) \otimes w + (-1)^{|X||v|}v \otimes (\nabla_Xw),\\
(\nabla_X f)(v)&=\nabla_X(f(v))-(-1)^{|f||X|}f(\nabla_X v).
\end{align*}
The curvature $R^{\nabla}(X,Y)$ is defined by 
\[\nabla_X\nabla_Y-(-1)^{|X||Y|}\nabla_Y \nabla_X-\nabla_{[X, Y]}.\]
A connection $\nabla$ is called {\it flat} if $\nabla$ is an endomorphism of a graded Lie algebra, i.e.,
$R^\nabla=0.$

Let $L$ be an $\LI$ algebra over $R$ and 
$\nabla$ be a connection on the graded $R$ module $L$.
Then $\nabla$ naturally gives connections on $EL$ and $\End_R(EL)$, which are also denoted by $\nabla.$
A connection $\nabla$ is called an {\it $\LI$ connection} if 
$\nabla_X (\,\widehat{\ell}\,)=0$ for any $X \in \der_{\ke}(R)$.
Explicitly, this equality is written as follows:
\[\nabla_X(\ell_k(y_1, \dots, y_k))=
  (-1)^{|X|} \sum_{i=1}^k (-1)^{(|y_1|'+\cdots+|y_{i-1}|')|X| }\ell_k(y_1, \dots, \nabla_Xy_i, \dots y_k).\] 
  
For a connection $\nabla$ on $L$, we define a morphism
$$
\wt{\nabla} : \der_\ke(R) \to \End_{\ke[[z]]}(\wt{L})
$$ 
by 
\begin{equation}\label{eq:Lconn}
\wt{\nabla}_X(y+\epsilon y'):=\nabla_X(y)+(-1)^{|X|}\epsilon \nabla_X(y').
\end{equation}
Note that $\nabla_X$ linearly extends as a $\ke[[z]]$ module morphism. 

\begin{definition}\label{def:CHconn}
Let $\wt{M}$ be a CH module over an $\LI$ algebra $L$ and 
$\nabla$ be an $\LI$ connection on $L$. 
A {\it CH connection on $\wt{M}$} is a morphism
$$
\wt{\nabla} : \der_\ke(R) \to \End_{\ke[[z]]}(\wt{M})
$$ 
such that $\wt{\nabla}$ is a connection on 
$\wt{M}$ regarded as a graded $R$ module and satisfies the following relations:
\begin{align*}
  \wt{\nabla}_X(\ell^{\wt{M}}_k(y_1, \dots, y_k|m))&=
  (-1)^{|X|} \sum_{i=1}^k (-1)^{(|y_1|'+\cdots+|y_{i-1}|')|X| }
  \ell^{\wt{M}}_k(y_1, \dots, \wt{\nabla}_Xy_i, \dots, y_k|m)\\
  &+(-1)^{(1+|y_1|'+\cdots+|y_{k}|')|X|}\ell^{\wt{M}}_k(y_1, \dots, y_k|\wt{\nabla}_Xm)
\end{align*}
for $y_i \in \wt{L}$ $($not $L[[z]])$. 
Here $\wt{\nabla}$ in the first term on the right hand side is the connection induced by the $\LI$ connection $\nabla$ on $L$ as in \eqref{eq:Lconn}.
\end{definition}    
\begin{remark}
For a connection $\wt{\nabla}$ on $\wt{M}$,
we define a connection on $\wt{M}[1]$ by 
$(-1)^{\vert X \vert} \wt{\nabla}$. We use the same symbol by abuse of notation.
%
\end{remark}

\begin{definition}\label{def:chhompresconn}
Let $(\wt{M},\wt{\nabla})$ and $(\wt{N},\wt{\nabla})$ be CH modules 
with CH connections over an $\LI$ algebra $L$ with an $\LI$ connection $\nabla$.  
We say that a CH module morphism $\{f_k\}_{0 \le k}$ from $\wt{M}$ to $\wt{N}$ {\it preserves the CH connections} if 
$$
\wt{\nabla}_X \circ \check{f} = \check{f} \circ \wt{\nabla}_X  
$$
for any $X \in \der_{\ke}(R)$.
Here $\check{f} \in \Hom^{0}_{R[[z]]}(E\wt{L}\otimes\wt{M}[1], E\wt{L}\otimes\wt{N}[1])$ is  
the comodule morphism corresponding to $\{f_k\}_{0 \le k}$ 
and $\wt{\nabla}_X$ are the $\ke[[z]]$ linear morphisms induced from the $\LI$ connection
and CH connections.
\end{definition}

\subsection{Maurer-Cartan elements}\label{MC}
Let $L$ be a normed $\LI$ algebra over a graded complete normed ring $R.$
Let $\gamma \in L^1 \cong L[1]^0.$
If $\|\gamma\| <1,$ then we can define 
     \[\e^\gamma:=
     \sum_{k=0}^\infty\frac{1}{k!} \overbrace{\gamma \hodot \cdots \hodot \gamma}^k \in 
     \widehat{E}L^0.\]
This is a group like element, i.e.,
$\Delta \e^\gamma=\e^\gamma \hotimes \e^\gamma$ and $\epsilon(\e^\gamma)=1.$
Moreover, the morphism $\e^\gamma \cdot$ is an isomorphism with the inverse 
$\e^{-\gamma} \cdot.$
\begin{definition}
$\gamma \in L^1$ is called a {\it Maurer-Cartan element} of $L$ if $\|\gamma\| < 1$ and $\gamma$ satisfies 
$\hell(\e^\gamma)=0,$ i.e., 
\begin{align*}
\sum_{k=1}^\infty \ell_k(\gamma, \dots, \gamma)/k!=0. 
\end{align*}
This equation is called a {\it Maurer-Cartan equation}.
\end{definition}
For a Maurer-Cartan element $\gamma$ of $L$, we can twist $\LI$ structure of $L$ by 
$$
\hell^\gamma:=\e^{-\gamma} \hell \e^\gamma.
$$
Explicitly, the corresponding set of morphisms $\{\ell^\gamma_k\}_{1 \le k}$ is written as follows:
\[\ell^\gamma_k(y_1, \dots y_k):=
  \sum_{i=0}^\infty\frac{1}{i!}\ell_{k+i}(\overbrace{\gamma, \dots, \gamma}^i, y_1, \dots, y_k).\]
We note that the Maurer-Cartan equation for $\gamma$ implies  ``$\ell^\gamma_0=0$''. 
By construction, we easily see that $\{\ell^\gamma_k\}_{1 \le k}$ gives an $\LI$ algebra structure. 
To simplify notation, this $\LI$ algebra is denoted by $L^\gamma.$

Let $L'$ be another normed $\LI$ algebra and 
$\{f_k\}_{0 \le k}$ be a normed $\LI$ morphism from $L$ to $L'.$
For $\gamma \in L^1$ with $\|\gamma\|<1$, we see that $\e^f(\e^\gamma)$ is a group like element.    
Hence there exists a unique element $f_*\gamma \in EL'^0$ such that 
$\|f_*\gamma\|<1$ and $\e^f(\e^\gamma)=\e^{f_*\gamma}.$
Explicitly, we have 
\[f_*\gamma=\sum_{k=1}^\infty\frac{1}{k!}f_k(\gamma, \dots, \gamma).\]
Moreover, if $\gamma$ is a Maurer-Cartan element, 
then $f_*\gamma$ is also a Maurer-Cartan element. 
Set $\e^{f, \gamma}:=\e^{-f_*\gamma}\e^f\e^{\gamma},$
then we easily see that $\e^{f, \gamma}$ is a coalgebra morphism with
$\e^{f, \gamma}\circ \eta_{EL} =\eta_{EL'}$. 
(Here $\eta_{EV}$ denotes a unit of $EV$. See 
the beginning of Section \ref{slinf}.)
The corresponding set of morphisms is denoted by $\{f^\gamma_k\}_{0 \le k}.$
If $\gamma$ is a Maurer-Cartan element, then $\{f^\gamma_k\}_{0 \le k}$ gives a normed $\LI$ morphism from
$L^\gamma$ to $L'^{f_*\gamma}.$

We next consider twists of normed CH modules.
Assume $\|R\|, \|L\| \le 1$.
Then $\wt{L}$ is also a normed $\LI$ algebra. 
If $\gamma$ is a Maurer-Cartan element of $L$, then we easily see that $\gamma$ is also a Maurer-Cartan element of $\wt{L}.$ 
Let $(\wt{M}, \{\ell^{\wt{M}}_k\}_{0 \le k})$ be a normed CH module over $L.$
By the left multiplication, $\e^\gamma$ gives an automorphism of 
$\widehat{E}\wt{L} \hotimes \wt{M}[1].$ 
We define $\{\ell^{\wt{M}, \gamma}_k\}_{0 \le k}$ by
$$
\widehat{\ell}^{\wt{M}, \gamma}:=\e^{-\gamma} \widehat{\ell}^{\wt{M}} \e^{\gamma}.
$$
Note that $\e^{-\gamma} \widehat{\ell}^{\wt{M}} \e^{\gamma}$ is a contractive coderivation of 
$\widehat{E}\wt{L}^\gamma \hotimes \wt{M}[1]$.
\begin{proposition}
For a Maurer-Cartan element $\gamma$ of $L$, the pair $(\wt{M}, \{\ell^{\wt{M}, \gamma}_k\})$ is a normed CH module over the normed $\LI$ algebra $L^\gamma.$
\end{proposition}
\begin{proof}
By the definition, we have 
$\widehat{\ell}^{\wt{M}, \gamma} \circ \widehat{\ell}^{\wt{M}, \gamma}=
   \e^{-\gamma} \widehat{\ell}^{\wt{M}} \circ \widehat{\ell}^{\wt{M}} \e^\gamma.$
Since the automorphism $\e^\gamma $ preserves the filtration on 
$\widehat{E}\wt{L}^\gamma \hotimes \wt{M}[1],$ we see that
$\widehat{\ell}^{\wt{M}, \gamma} \circ \widehat{\ell}^{\wt{M}, \gamma}(F^a) \subset F^{a+2}.$  
\end{proof}
This normed CH module is denoted by 
$\wt{M}^\gamma$. 
Accordingly, the corresponding maps in 
\eqref{ch3.2.2}, \eqref{ch3.3} are denoted by 
\begin{equation}\label{dgamma}
d^{\gamma}, ~ {\LL}_y^{\gamma}, ~ \io_{y_1}^{\gamma}, ~ \rho_{y_1, y_2}^{\gamma}.
\end{equation}
\par
Let $\wt{N}$ be another normed CH module over $L$ and
$\{f_k\}_{0 \le k}$ be a normed CH morphism from  $\wt{M}$ to $\wt{N}$.
Set 
\begin{equation}\label{fgamma}
\check{f}^\gamma:=\e^{-\gamma} \check{f}\e^\gamma.
\end{equation}
This is a contractive comodule morphism and the corresponding set of morphisms is denoted by $\{f^\gamma_k\}_{0 \le k}$.
\begin{proposition}
$\{f^\gamma_k\}_{0 \le k}$ is a normed CH morphism from $\wt{M}^\gamma$ to $\wt{N}^\gamma.$
\end{proposition}
\begin{proof}
By definition, we have
\[\widehat{\ell}^{\wt{N}, \gamma} \circ \check{f}^\gamma-\check{f}^\gamma \circ \widehat{\ell}^{\wt{M}, \gamma}=
\e^{-\gamma} (\widehat{\ell}^{\wt{N}} \circ \check{f}-\check{f} \circ \widehat{\ell}^{\wt{M}})\e^\gamma,\]
which implies the proposition.  
\end{proof}
For the normed CH morphism $\{f^\gamma_k\}_{0 \le k}$, the corresponding maps in \eqref{chmor3.7} are denoted by 
\begin{equation}\label{chmor5.4}
F^\gamma_y, ~F_y^{\epsilon, \gamma}.
\end{equation}

\subsection{Getzler-Gauss-Manin connections}\label{subge}
In \cite{getcar}, Getzler constructed a connection on periodic cyclic homology of $A_\infty$ algebras.
Following Barannikov's idea outlined in \cite[Remark 3.3]{barqua},
Tsygan reformulated this connection in terms of $\LI$ modules in \cite{tsygau}.
We adapt Tsygan's reformulation and introduce similar connections on 
our normed CH modules.

In the rest of Section \ref{sec:conn}, let $\ke$ be a complete normed ring and $R$ be a graded complete normed algebra over $\ke$ unless otherwise mentioned.
Let $L$ be a normed $\LI$ algebra over $R$ with an $\LI$ connection $\nabla$. 
We assume that $\no{\ke}, \no{R}, \no{L} \le 1$. 
Let $\wt{M}$ be a normed CH module over $L$ with a CH connection $\wt{\nabla}$.
Let $\gamma \in L^1$ be a Maurer-Cartan element. 
\begin{definition}\label{defggm}
We define a morphism $z\nabla^\ggm : \der_\ke(R) \to \End_{\ke[[z]]}(\wt{M})$ as follows:
\begin{align*}
z\nabla^\ggm_X:=z\wt{\nabla}_X-(-1)^{|X|}\io^\gamma_{\nabla_X\gamma}.
\end{align*}
This morphism is called a {\it Getzler-Gauss-Manin connection} (GGM connection for short).
\end{definition}
\begin{remark}
See \eqref{ch3.3} and Remark \ref{lsign} for the definition of $\io$. 
Since $\io$ is a morphism from $L\otimes \wt{M},$ 
the connection $\nabla$ in the term $\io^\gamma_{\nabla_X\gamma}$ should be considered as 
a connection on $L$ $($not a connection on $L[1]$$)$. 
\end{remark}
Since $\io^\gamma_{\nabla_X\gamma} \in \End^{2+|X|}_{R[[z]]}(\wt{M}),$         
we see that $z\nabla^\ggm$ satisfies the ``Leibniz rule'' in the following form: 
\[z\nabla^\ggm_X(rm)=z(Xr)m+(-1)^{|X||r|}rz\nabla^\ggm_Xm.\]  
Thus the GGM connection $z\nabla^\ggm$ itself is not a connection in the sense of
the beginning of Subsection \ref{5.1}, but later 
$\nabla^\ggm := (z\nabla^\ggm)/z$ will be regarded as a meromorphic connection and 
also called a GGM connection.      

\begin{example}\label{5.10}
({\bf Continued from Example \ref{hypcom}}.)
Let $A$ be a finite-dimensional graded vector space over a field $\ke$ and choose a homogeneous basis $\{T_0, \dots, T_m\}$. 
Assume that $A^e:=A((e))$ be a hypercommutative algebra over $\ke^e:=\ke ((e))$.
(See Example \ref{4.3} (1) for the notation.)
Choose formal variables $t_0, \dots, t_m$ with degrees $2-|T_0|, \dots, 2-|T_m|$ respectively.
Set $R:=\ke^e[[t]].$
Then $A^e[[t]]$ is naturally considered as a hypercommutative algebra over $R$ equipped with the trivial connection.
By Example \ref{hypcom}, $A^e[[t]]$ is a CH module over $A^e[[t]][1].$
We define a Maurer Cartan element $\gamma$ to be ${\bf t}=\sum_{i=0}^mT_i t_i$.
Set 
\[(T_{i_1}, \dots, T_{i_k})^\gamma:=
  \sum_{l=0}^\infty\frac{1}{l!}(\overbrace{\gamma, \dots, \gamma}^l, T_{i_1}, \dots, T_{i_k}), \ 
T_i *_{{\bf t}} T_j:=(T_i, T_j)^\gamma. \]
Then $A^e[[t]]$ with operators $(\cdot, \dots, \cdot)^\gamma$ is a hypercommutative algebra.
In this setting we see that 
$z\nabla^\ggm_{\frac{\partial}{\partial t_i}}=z\frac{\partial}{\partial t_i}+T_i *_{\bf t}$.
This gives a reinterpretation of Dubrovin's connection in \cite{dubint}. 
\end{example}

To show that GGM connection defines a morphism on cohomology, we use the following lemmas:
\begin{lemma}\label{5.6}
$\ell^{\gamma}_1(\nabla_X\gamma)=0.$
\end{lemma}
\begin{proof}
Since $|\gamma|=1,$ we easily see that 
$\nabla_X \ell_k(\gamma, \dots, \gamma)=
 (-1)^{|X|}k\ell_k(\gamma, \dots, \gamma, \nabla_X\gamma)$.
 Hence we have 
 \[\ell^{\gamma}_1(\nabla_X\gamma)=
  \sum_{k=0}^\infty \ell_{1+k}(\gamma, \dots, \gamma, \nabla_X\gamma)/k!=
 (-1)^{|X|}\nabla_X \left( \sum_{k=1}^\infty \ell_k(\gamma, \dots, \gamma)/k! \right)=0.\]
\end{proof}
\begin{lemma}\label{5.7}  
$[\wt{\nabla}_X, \ell^{\wt{M}, \gamma}_0]=
   (-1)^ {|X|}\ell^{\wt{M}, \gamma}_1({\nabla}_X\gamma|\,\cdot\,).$
\end{lemma}
\begin{proof}
We see that
\begin{align*} [\wt{\nabla}_X, \ell^{\wt{M}, \gamma}_0]&=
    [\wt{\nabla}_X, \sum_{k=0}^\infty \ell^{\wt{M}}_k(\gamma, \dots, \gamma|\,\cdot\,)/k!\,]=
    (-1)^{|X|}\sum_{k=0}^\infty 
                \ell^{\wt{M}}_{1+k}(\gamma, \dots, \gamma, {\nabla}_X \gamma|\, \cdot\,)/k! \\
                &=(-1)^ {|X|}\ell^{\wt{M}}_1({\nabla}_X\gamma|\,\cdot\,).
    \end{align*}
\end{proof}
\begin{proposition}\label{ggmcoh}
$[d^\gamma, z\nabla^\ggm_X]=0$.
\end{proposition}
\begin{proof}
By Equation (\ref{ch3.4}) and Lemma \ref{5.6}, we have 
$[d^{\gamma}, \io^\gamma_{\nabla_X\gamma}]+z\LL^\gamma_{\nabla_X\gamma}=0.$
By Lemma \ref{5.7}, we have 
$[d^\gamma, z\wt{\nabla}_X]=
(-1)^{1+|X|}z\LL^\gamma_{\nabla_X\gamma}.$
This proposition follows from these equations.
\end{proof}
\subsection{Euler connections}\label{subeu}
In \cite[\S 2.2.5]{kkp1}, Katzarkov-Kontsevich-Pantev introduced a connection in the ``$u$-direction'' on a cyclic homology (see also \cite{shknon}). 
Inspired by their construction
(but different from theirs as noted in Remark \ref{rmk:KKP}), we extend a GGM connection 
to a ``connection in the $z$-direction''.
In different contexts, 
similar constructions have been considered by many people (e.g., \cite[\S 4.1]{irimir}, \cite{seicon} 
and see also Remark \ref{rmk:GPSS}).

In \S\ref{subeu}, we continue to use the same notations as in \S\ref{subge}.
We note that 
\[\der_\ke(R) \cong \left\{ X \in \der_\ke(R[[z]])\relmiddle{|}Xz=0\right\} \subset \der_\ke(R[[z]])\]
and $\der_\ke(R[[z]])=\der_\ke(R) \oplus \langle \frac{d}{dz} \rangle.$
Let $\deg \in \End^0_\ke(\wt{M})$ be the degree operator, i.e., $\deg m := |m|m.$ 
\begin{definition}\label{def:Eu}
We define a morphism $z^2\nabla^{\eul} : \der_\ke(R[[z]]) \to \End_\ke(\wt{M})$ as follows:
\begin{align*}
z^2 \nabla^\eul_X&:=z^2\nabla^\ggm_X ,  \\
z^2 \nabla^\eul_{\frac{d}{dz}}&:=\frac{z}{2}\deg-z\nabla^\ggm_E,
\end{align*}
where $X \in \der_\ke(R)$ and $E$ is the Euler vector of $R.$
This morphism is called an {\it Euler connection}.
\end{definition}
Since $\deg(rm)=2(Er)m+2(z\frac{d}{dz}r)m+r\deg m$ and
$z\nabla^\ggm_E rm=z(Er)m+rz\nabla^\ggm_Em,$
we see that $\nabla^\eul_\frac{d}{dz}$ satisfies the following ``Leibniz rule'':
\[z^2\nabla^\eul_{\frac{d}{dz}}(rm)=(z^2\frac{d}{dz}r)m+rz^2\nabla^\eul_\frac{d}{dz}m.\]
\begin{proposition}[{cf. \cite[\S 2.2]{shknon}}]
$[z^2 \nabla^\eul_\frac{d}{dz}, d^\gamma]=\frac{z}{2}d^\gamma.$
\end{proposition}
\begin{proof}
Since $d^\gamma$ is a degree one operator, we see $[\deg, d^\gamma]=d^\gamma.$
Hence this proposition follows from Proposition \ref{ggmcoh}.
\end{proof}
From this lemma, we easily see that $z^2 \nabla^\eul$ descends to a morphism on cohomology.
Set $\wt{M}((z)):=\wt{M} \otimes_{R[[z]]}R((z)).$
By abuse of notation, we consider $\nabla^\ggm, \nabla^\eul$ as meromorphic connections (on complexes or cohomology).
\begin{remark}
The Euler connection has an irregular singularity at $z=0,$ but formal in the $z$-direction. 
Hence we can not consider the ``Stokes structure'' of the Euler connection.  
\end{remark}
\begin{proposition}\label{flat}
If $\nabla^\ggm$ is flat on cohomology, then $\nabla^\eul$ is also flat on cohomology.  
\end{proposition}
\begin{proof}
By the assumption, we have 
\begin{align*} [\nabla^\eul_\dz, \nabla^\eul_X]&=
[\frac{1}{2z}\deg, \nabla^\ggm_X]-[\frac{1}{z} \nabla^\ggm_E, \nabla^\ggm_X] \\
&=\frac{1}{2z}|X|\nabla^\ggm_X-\frac{1}{z}\nabla^\ggm_{[E, X]}-\frac{1}{z}R^{\nabla^\ggm}(E, X) \\
&=-\frac{1}{z}R^{\nabla^\ggm}(E, X),\end{align*}
where we use $[\deg, \nabla^\ggm_X]=|X|\nabla^\ggm_X$ and $[E, X]=\frac{1}{2}|X|X.$
Hence we have 
\begin{align*}[\nabla^\eul_\dz, \nabla^\eul_\dz]&=
 [\nabla^\eul_\dz, \frac{1}{z}\left(\frac{1}{2}\deg-\nabla^\eul_E\right)]\\
 &=
 -\frac{1}{z}\nabla^\eul_\dz
 +\frac{1}{z}[\nabla^\eul_\dz, \frac{1}{2}\deg-\nabla^\eul_E] \\
 &=-\frac{1}{z}\nabla^\eul_\dz+\frac{1}{z}\nabla^\eul_\dz+
 \frac{1}{z^2}R^{\nabla^\ggm}(E, E)\\
 &=\frac{1}{z^2}R^{\nabla^\ggm}(E, E).\end{align*}
This proposition follows from these equations.       
\end{proof}
\begin{remark}
By a direct calculation, we see that
\begin{equation}\label{eq:curv}
\aligned
R^{\nabla^\ggm}(X, Y)&=R^{\wt{\nabla}}(X,Y)\\
&-\frac{1}{z}\left((-1)^{|X|+|Y|} \io^\gamma_{R^\nabla(X, Y)\gamma}
  +(-1)^{|Y|}\rho^\gamma_{\nabla_X\gamma,  \nabla_Y\gamma}-(-1)^{|X|+|X||Y|}
  \rho^\gamma_{\nabla_Y\gamma, \nabla_X\gamma}\right)\\
&+\frac{1}{z^2}(-1)^{|X|+|Y|}[\io^\gamma_{\nabla_X\gamma}, \io^\gamma_{\nabla_Y \gamma}].
\endaligned
\end{equation}
%
%
On the other hand, if the CH structure extends to an $\LI$ module structure$\mod \epsilon^3$, 
then the $\LI$ module relations mod $\epsilon^3$ implies
\begin{align*}
 &\ell^{\wt{M}, \gamma}_0\ell^{\wt{M}, \gamma}_2(\epsilon y_1, \epsilon y_2|m)
 +(-1)^{|y_1|+|y_2|}\ell^{\wt{M}, \gamma}_2(\epsilon y_1, \epsilon y_2|\ell^{\wt{M}, \gamma}_0m)\\
 +&(-1)^{|y_1|} \ell^{\wt{M}, \gamma}_1(\epsilon y_1|\ell^{\wt{M}, \gamma}_1(\epsilon y_2|m))
 +(-1)^{|y_1||y_2|+|y_2|} \ell^{\wt{M}, \gamma}_1(\epsilon y_2|\ell^{\wt{M}, \gamma}_1(\epsilon y_1|m))\\
 -&\ell^{\wt{M}, \gamma}_2(\epsilon \ell^\gamma_1 y_1, \epsilon y_2|m)
 -(-1)^{|y_1||y_2|}\ell^{\wt{M}, \gamma}_2(\epsilon \ell^\gamma_1 y_2, \epsilon y_1|m) \\
 +&\ell^{\wt{M}, \gamma}_2(zy_1, \epsilon y_2|m)
 +(-1)^{|y_1||y_2|}\ell^{\wt{M}, \gamma}_2(zy_2, \epsilon y_1|m)
 \\
 =&0,
\end{align*}
where $y_1, y_2 \in L[[z]].$ 
Set 
$\LL^\gamma _{y_1, y_2}:=
(-1)^{|y_1|+|y_2|}\ell^{\wt{M}, \gamma}_2(\epsilon y_1, \epsilon y_2|\, \cdot\,) \in \End(\wt{M}).$
Then the above formula is written as follows:
\begin{align*}
 [d, \LL^\gamma_{y_1, y_2}]+\LL^\gamma_{\delta y_1, y_2}+(-1)^{|y_1|}\LL^\gamma_{y_1, \delta y_2}
 =(-1)^{|y_1|}[\io^\gamma_{y_1}, \io^\gamma_{y_2}]+
 z\big{(}\rho^\gamma_{y_1, y_2}+(-1)^{|y_1||y_2|}\rho^\gamma_{y_2, y_1}\big{)}.
\end{align*}
Combing \eqref{eq:curv} with this formula, we find
that $\nabla^\ggm$ is flat on cohomology if $\nabla$  and $\wt{\nabla}$ are flat and the CH structure extends to an $\LI$ module structure$\mod \epsilon^3.$  
\end{remark}
\begin{remark}\label{5.19}
In Example \ref{5.20} below, we will use the following variant 
of the GGM connection.
Let $L$ be an $\LI$ algebra $($not normed$)$ over a graded ring $R$ 
and $\wt{M}$ a CH module over $L$, 
and let $\nabla$ be a connection
on $L$ (not necessary an $\LI$ connection)
and  
$\wt{\nabla} :  \der_\ke(R) \to \End_{\ke[[z]]}(\wt{M})$
a connection on $\wt{M}$ regarded as a graded $R$ module 
(not necessary a CH-connection).
Suppose that we have a degree one morphism
\[c_1 : \der_\ke(R) \to L, \]
for which $\nabla$ satisfies
\begin{equation}\label{eq:nabla}
(\nabla_X\ell_k)(y_1, \dots, y_k)=(-1)^{|X|}\ell_{k+1}(c_1(X), y_1, \dots, y_k)
\end{equation}
and $\wt{\nabla}$ satisfies 
\begin{equation}\label{eq:wtnabla}
(\wt{\nabla}_X\ell^{\wt{M}}_k)(y_1+\epsilon y'_1, \dots, y_k+\epsilon y'_k|m)=(-1)^{|X|}\ell^{\wt{M}}_{k+1}(c_1(X), y_1+\epsilon y'_1, \dots, y_k+\epsilon y'_k|m).
\end{equation}
Then the connections $\nabla$ and $\wt{\nabla}$ satisfy the properties in 
Lemma \ref{5.6} and Lemma \ref{5.7}. 
Thus in Definition \ref{defggm}, replacing the CH connection by the connection $\wt{\nabla}$ above and
$\nabla_X \gamma$ by $c_1(X)$, we can define a (variant of) GGM connection 
on cohomology. Similarly, we can define a (variant of) Euler connection as well in this setting.
\end{remark}
\begin{example}\label{5.20}
({\bf Continued from Example \ref{5.10}}.)
Suppose that there exists an element \[c_1 \in A^2 \subset A^{e, 2}\] 
with
\begin{equation}\label{eq:divaxm}
(c_1, x_1, \dots, x_k)=e\frac{d}{de}(x_1, \dots, x_k).
\end{equation}
We define a degree one $R$ module morphism \[c_1 :\der_\ke(R) \to A^e[1]\]
by $c_1(e\frac{d}{de}):=c_1$ and $\ c_1(\frac{\partial}{\partial t_i}):=(-1)^{|T_i|}T_i.$
Then the trivial connections $\nabla$ and $\wt{\nabla}$ satisfy the condition of Remark \ref{5.19}.
Set 
$$
E:=c_1+\sum_{i=0}^m (-1)^{|T_i|}(\frac{2-|T_i|}{2})T_i t_i \in A^e[[t]].
$$
Then we easily see that 
\[\nabla^\eul_{\frac{\partial}{\partial t_i}}=\frac{\partial}{\partial t_i}+\frac{1}{z}T_i *_{\bf t}, \  
\nabla^\eul_{\dz}=\dz+\frac{\deg}{2z}-\frac{E*_{\bf t}}{z^2},\]
where ${\bf t} =\sum_{i=0}^m T_it_i$ and $\deg$ is the degree operator of $A$.
\end{example}
Now let $\wt{N}$ be another normed CH module equipped with a CH connection and $\{f_k\}_{0 \le k}$ be a morphism of normed CH modules from $\wt{M}$ to $\wt{N}$ preserving 
the CH connections on $\wt{M}$ and $\wt{N}$ in the sense of Definition \ref{def:chhompresconn}. 
Then recalling \eqref{dgamma}, \eqref{fgamma}, \eqref{chmor5.4},
we have 
\begin{proposition}\label{prop5.21}
$$
\nabla^\ggm_X \circ f^\gamma_0-f^\gamma_0 \circ \nabla^\ggm_X
 = 
     (-1)^{|X|+1} d^\gamma \circ 
     F^{\epsilon,\gamma}_{\nabla^\ggm_X\gamma}
     -F^{\epsilon,\gamma}_{\nabla^\ggm_X \gamma}  \circ d^\gamma.
$$
\end{proposition}
\begin{proof}
Since the morphism preserves the CH connections, we see that 
$$
\wt{\nabla}_X \circ f_0^\gamma 
= f_0^\gamma \circ \wt{\nabla}_X-(-1)^{|X|}
F^\gamma_{\nabla_X \gamma}.
$$
By Equation \ref{chmo} and Lemma \ref{5.6}, we have 
\[
zF^\gamma_{\nabla_X\gamma}+\io^\gamma_{\nabla_X\gamma} \circ f^\gamma_0=
   f^\gamma_0 \circ \io^\gamma_{\nabla_X\gamma}+d^\gamma \circ 
F^{\epsilon,\gamma}_{\nabla_X \gamma}
   -(-1)^{|X|+1} F^{\epsilon,\gamma}_{\nabla_X \gamma} \circ d^\gamma.\]
The proposition follows from these equations.
\end{proof}
Since $f^\gamma_0$ is degree zero, we obtain $\deg \circ f^\gamma_0=f^\gamma_0 \circ \deg.$   
Hence we have the following:
\begin{corollary}\label{cor5.22}
$$
\nabla^\eul_{\frac{d}{dz}}\circ f^\gamma_0-f^\gamma_0 \circ \nabla^\eul_{\frac{d}{dz}}=
   \frac{1}{z}\left(d^\gamma \circ 
   F^{\epsilon,\gamma}_{\nabla^\eul_E\gamma}
     -F^{\epsilon,\gamma}_{\nabla^\eul_E \gamma}  \circ d^\gamma\right).
$$ 
\end{corollary}
Proposition \ref{prop5.21} and Corollary \ref{cor5.22} imply that $f^\gamma_0$ intertwines the Euler connections on cohomology. 
\begin{remark}\label{rem:varComp}
({\bf Continued from Remark \ref{5.19}}.)
We consider the situation of 
the variants of the GGM connection and the Euler connection in  
Remark \ref{5.19}. Instead of assuming that 
a morphism $\{f_k\}_{0 \le k}$ of normed CH modules from $\wt{M}$ to $\wt{N}$ preserves 
the CH connections, we assume that $\{f_k\}_{0 \le k}$ satisfies 
$$
(\wt{\nabla}_X \circ f_k -f_k \circ \wt{\nabla}_X)
(y_1+\epsilon y'_1, \dots, y_k+\epsilon y'_k|m) = 
f_{k+1}(c_1(X), y_1+\epsilon y'_1, \dots, y_k+\epsilon y'_k|m) 
$$
for any $X \in \der_\ke(R)$.
Here $\wt{\nabla}$ in the first term is the connection on $\wt{N[1]}$ as in \eqref{eq:wtnabla} 
and one in the second term is the connection on 
$\wt{L}[1]^{\odot k}\otimes \wt{M}[1]$ induced by the connections $\nabla, \wt{\nabla}$ on $L, \wt{M}$ 
in \eqref{eq:nabla}, \eqref{eq:wtnabla} respectively.
Then Proposition \ref{prop5.21} for the variant of the GGM connection $\nabla^\ggm$
holds 
and Corollary \ref{cor5.22} for the variant of the Euler connection $\nabla^\eul$ also holds.
\end{remark}

\section{CH structures on Hochschild invariants}\label{sec:CH}
\subsection{CH structures on Hochschild invariants}
In \S 6.1., we give an explicit formula of a CH structure on Hochschild invariants of an $A_\infty$ category (cf. \cite{dalope}, \cite{tsygau}).
We mainly follow \cite{getcar} for the definitions of operators on Hochschild invariants.
Let $R$ be a graded ring (not normed).
Let $\A$ be a set of objects $\ob\A$ with a set of morphisms  $\Hom_\A(X, Y)$ for each $X,Y \in \ob\A,$
where $\Hom_\A(X, Y)$ is a graded $R$ module.
We assume $\Hom_\A(X, X)$ is equipped with a degree preserving $R$ module morphism 
$\eta : R \to \Hom_\A(X, X)$ for each $X \in \ob\A.$
This morphism $\eta$ is called a unit and $\eta(1)$ is denoted by $\mathbbm{1}.$
Set \[\A(X_0, \dots, X_k):=\Hom_\A(X_0, X_1)[1] \otimes \cdots \otimes \Hom_\A(X_{k-1}, X_k)[1].\]  
The Hochschild cochain complex and the Hochschild chain complex are defined by 
\begin{align*}
 \hoc^l&:=\!\!\!\!\prod_{X_0, \dots, X_k \in \ob\A}\!\!\!\! 
 \Hom_R^l\big(\A(X_0, \dots, X_k), \Hom_\A(X_0, X_k)\big), 
 &\hoc:=\bigoplus_{l \in \Z} \hoc^l, \\
 \hoh^l&:=\bigoplus_{X_0, \dots X_k \in \ob\A} 
 \big(\Hom_\A(X_0, X_1) \otimes \A(X_1, \dots, X_k, X_0)\big)^l, 
 &\hoh:= \bigoplus_{l \in \Z} \hoh^l.
\end{align*}
For $\varphi, \psi \in \hoc[1],$ we define a composition $\varphi \circ \psi$ by 
\begin{align*}
 \varphi \circ \psi(x_1, \dots, x_k):=
 \sum_ {0 \le i \le j \le k} (-1)^\# \varphi(x_1, \dots, x_i, \psi(x_{i+1}, \dots, x_j), x_{j+1}, \dots, x_k),
\end{align*}
where $x_i \in \Hom_A(X_{i-1}, X_i)[1]$ and the sign $\#$ is $|\psi|'(|x_1|'+\cdots+|x_i|').$
Set \[[\varphi, \psi]:=\varphi \circ \psi- (-1)^{|\varphi|'|\psi|'}\psi \circ \varphi.\] 
For $\X=x_0 \otimes \cdots \otimes x_k \in \A(X_0, \dots, X_k, X_0) \subset \hoh[1],$
we define $\LL_{\varphi}(\X)$ by
\begin{align*}
\LL_\varphi(\X):=&\sum_{0 \le i \le j \le k} (-1)^{\#_1}x_0 \otimes \cdots \otimes x_i \otimes \varphi(x_{i+1}, \dots, x_j) \otimes
  x_{j+1} \cdots \otimes x_k \\
 +&\sum_{0 \le i \le j \le k} (-1)^{\#_2}\varphi(x_{j+1}, \dots, x_k, x_0, \dots, x_i) \otimes x_{i+1} \otimes \cdots \otimes x_j,
\end{align*}
where 
$$
\aligned
\#_1 & :=|\varphi|'(|x_0|'+\cdots+|x_i|'), \\
\#_2 & :=(|x_0|'+ \cdots + |x_j|')(|x_{j+1}|'+ \cdots +|x_k|').
\endaligned
$$
\begin{proposition}[see, e.g., \cite{getcar}, {\cite[Proposition 3.1.]{iwaper}}]\label{ch0}
The shifted Hochschild cochain complex $(\hoc[1], [\cdot,\cdot])$ is a graded Lie algebra and $(\hoh[1], \LL)$ is a graded Lie module over it.
\end{proposition}
Following \cite[\S 2]{getcar}, we define a morphism $\rho$ by
\begin{align*}
\rho_{\varphi, \psi}:=\!\sum_{0 \le i \le j \le s \le t \le k}\!(-1)^\#
  \varphi(x_{j+1}, \dots, x_s, \psi(x_{s+1}, \dots, x_t),x_{t+1}, \dots, x_k, x_0, \dots, x_i) 
   \otimes x_{i+1} \otimes \cdots \otimes x_j,
\end{align*}
where $\#=|\psi|'(|x_{j+1}|'+ \cdots + |x_s|')+(|x_0|'+ \cdots + |x_j|')(|x_{j+1}|'+ \cdots + |x_k|').$
The next proposition will be proved in \S \ref{proof}.
\begin{proposition}\label{ch1}
\begin{align*}
\rho_{[\varphi_1, \varphi_2], \psi}-\rho_{\varphi_1,[\varphi_2, \psi]}+
 (-1)^{|\varphi_1|'|\varphi_2|'}\rho_{\varphi_2, [\varphi_1, \psi]}=
 [\LL_{\varphi_1}, \rho_{\varphi_2, \psi}]-(-1)^{|\varphi_1|'|\varphi_2|'}
 [\LL_{\varphi_2}, \rho_{\varphi_1, \psi}].
\end{align*}
\end{proposition}
The reduced Hochschild cochain complex $\rhoc$ is defined by 
\[\rhoc:=\left\{ \varphi \in \hoc \relmiddle| \varphi( \dots, \mathbbm{1}, \dots)=0\right\}\]
and the reduced Hochschild chain complex $\rhoh$ is defined by the relation
\[x_0 \otimes x_1 \otimes \cdots x_{i-1} \otimes \mathbbm{1}
   \otimes x_{i+1} \otimes \cdots \otimes  x_k=0 \ \ (1 \le i \le k).\]
Then $\rhoc[1]$ is a sub graded Lie algebra and 
$\LL$ naturally induces a graded Lie module structure on $\rhoh[1],$ which is also denoted by $\LL.$
We define the following morphisms:
\begin{align*}
B(\X)&:=\sum_{0 \le i \le k}(-1)^{\#_1} \mathbbm{1}\otimes x_{i+1} \otimes \cdots \otimes x_k \otimes x_0 \otimes \cdots \otimes x_i, \\
B^1_\varphi(\X)&:=\sum_{0 \le i \le j \le s \le k}(-1)^{\#_1+\#_2}
  \mathbbm{1}\otimes x_{i+1} \otimes \cdots \otimes x_j \otimes 
  \varphi(x_{j+1}, \dots, x_s) \otimes \cdots \otimes x_k\otimes x_0 \otimes \cdots \otimes x_i,
\end{align*}
where
$\#_1=(|x_0|'+ \cdots + |x_i|')(|x_{i+1}|'+ \cdots + |x_k|')$ and
$\#_2=|\varphi|'(|x_{i+1}|'+ \cdots + |x_j|').$
By the definition, 
we can check the following (e.g.,\cite[Definition 2.1 and Theorem 2.2]{getcar}):
\begin{align}
[B, B]=0, \label{ch6.1} \\ 
[B, \LL_\varphi]=0, \label{ch6.2}\\ 
[B, B^1_\varphi]=0. \label{ch6.3}
\end{align}
The next proposition will be proved in \S \ref{proof}.
\begin{proposition}\label{ch2}
$
[B, \rho_{\varphi, \psi}]+B^1_{[\varphi, \psi]}-(-1)^{|\varphi|'}[\LL_\varphi, B^1_\psi]=0
$
\end{proposition}
Now we recall the definition of $A_\infty$ structures.
\begin{definition}
A {\it strictly unital curved $A_\infty$ structure} on $\A$ is an element $m \in \hoc[1]^1$ with
\[ [m, m]=0, \ \ 
   m_2(\mathbbm{1}, x)=(-1)^{|x|}m_2(x, \mathbbm{1}), \ \ 
   m_k(\dots, \mathbbm{1}, \dots)=0 \ \ (k \not= 2).\]
Here the length $k$ part of $m$ is denoted by $m_k$.
In case $m_0=0$, $m$ is called a 
{\it strictly unital $A_\infty$ structure} on $\A$.
\end{definition}
Let $m$ be a strictly unital curved $A_{\infty}$ structure on $\A$ and set
\[b:=\LL_m, \ \ b^1_\varphi:=\rho_{m, \varphi}, \ \ \delta:=[m, \,\cdot\,].\]
Then these morphisms naturally induce morphisms on reduced Hochschild complexes, which are denoted by the same symbols.
Note that $(\delta, [\cdot,\cdot])$ makes $\rhoc[1]$ into a DGLA and $(b, \LL)$ makes $\rhoh[1]$
into a DGLA module over the DGLA $\rhoc[1]$.
By Proposition \ref{ch1} and $[m, m]=0$, we have
\begin{align}
\rho_{\delta\varphi, \psi}-b^1_{[\varphi, \psi]}+(-1)^{|\varphi|'}\rho_{\varphi, \delta\psi}&=
[b, \rho_{\varphi, \psi}]-(-1)^{|\varphi|'}[\LL_\varphi, b^1_\psi] \label{ch6.4},\\ 
b^1_{\delta \psi}+[b, b^1_\psi]&=0.\label{ch6.5}
\end{align} 
We use the following identity
(see, e.g., \cite[Definition 2.1 and Theorem 2.2]{getcar}) 
\begin{align} \label{ch6.6}
[B, b^1_\varphi]+[b, B^1_\varphi]+B^1_{\delta\varphi}+\LL_\varphi=0.
\end{align}

Now we set 
$$
L:=\rhoc[1], \quad \wt{M}:=\rhoh[1][[z]].
$$
We note that the natural $z$-linear extensions of morphisms $\LL, \rho,$ etc. are denoted by the same symbols. 
For $\varphi \in L[[z]]$
we put 
$$
d:=b+zB, \quad \io_\varphi:=b^1_\varphi+zB^1_\varphi.
$$
Then 
combining Propositions \ref{ch0}, \ref{ch1}, \ref{ch2}, 
and the equalities \eqref{ch6.1}, \eqref{ch6.2}, \eqref{ch6.3}, \eqref{ch6.4}, \eqref{ch6.5}, (\ref{ch6.6})
we obtain the following theorem: 
\begin{theorem}\label{thm:68}
The morphisms $\LL, \io, \rho$ give a CH structure on $\wt{M}$ over the DGLA $L.$
\end{theorem}
\begin{proof}
Using Proposition \ref{ch0}, the equalities \eqref{ch6.1}, \eqref{ch6.2}, and $[m, m]=0$, we see that 
$(\wt{M}, d, \LL)$ is a DGLA module over the DGLA 
$L[[z]],$ where the DGLA structure on $L$ is $z$-linearly extended to $L[[z]].$
By Proposition \ref{chprop3.7}, it is sufficient to show the equalities \eqref{ch3.4}), \eqref{ch3.5}, (\ref{ch3.6}). 
We note that the equality \eqref{ch3.6} is a direct consequence of Proposition \ref{ch1}.
We check \eqref{ch3.4} and \eqref{ch3.5}.
Since
\begin{align*}
&[d, \io_\varphi]+\io_{\delta \varphi}+z\LL_\varphi 
  =[b, b^1_\varphi]+b^1_{\delta \varphi}
  +z([b, B^1_\varphi]+[B, b^1_\varphi]+B^1_{\delta \varphi}+\LL_\varphi)
  +z^2[B, B^1_\varphi],
\end{align*}
the equality \eqref{ch3.4} follows from \eqref{ch6.3}, \eqref{ch6.5}, \eqref{ch6.6}.
Since
\begin{align*}  
&\io_{[\varphi, \psi]}\!-(-1)^{|\varphi|'}\![\LL_{\varphi}, \io_{\psi}]+
        [d, \rho_{\varphi, \psi}]-\rho_{\delta \varphi, \psi}-(-1)^{|\varphi|'}\!\rho_{\varphi, \delta \psi}\\
=&b^1_{[\varphi, \psi]}\!-(-1)^{|\varphi|'}\![\LL_\varphi, b^1_\psi]+[b, \rho_{\varphi, \psi}]
  -\rho_{\delta\varphi, \psi}-(-1)^{|\varphi|'}\!\rho_{\varphi, \delta\psi}
+\!z(B^1_{[\varphi, \psi]}\!-\!(-1)^{|\varphi|'}\![\LL_\varphi, B^1_\psi]+[B, \rho_{\varphi, \psi}]),         
\end{align*}
the equality \eqref{ch3.5} follows from \eqref{ch6.4} and Proposition \ref{ch2}.
\end{proof}
\begin{definition}
The cohomology of $(\rhoh[[z]], b+zB)$ is called a {\it negative cyclic homology} and the cohomology of 
$(\rhoh((z)), b+zB)$ is called a {\it periodic cyclic homology}.
\end{definition}

\begin{remark}\label{filainf}
Let $\ke$ be an ungraded field and $R$ be a graded algebra over $\ke^e =\ke ((e))$. 
Note that $\ke$ and $\ke^e$ are equipped with the trivial norms.
Assume that $R$ is equipped with a norm $\no{\cdot}$ such that $R$ is a graded complete normed algebra over $\ke^e$, i.e., $\no{\cdot}$ is a ring norm and a $\ke^e$ module norm. 
We also assume $\no{R} \le 1$.
Let $m$ be a strictly unital $A_{\infty}$ structure on $\A$ over $\ke$.
Then $m$ gives a DGLA structure on $\rhoc[1]$ and this DGLA structure naturally extends to 
$\rhoc[1] \widehat{\otimes}_\ke R,$ where $\rhoc[1]$ is considered as a graded $\ke$ module with the trivial norm.
In this note, we define a {\it gapped filtered graded $\ainf$ structure on $\A$ by a Maurer-Cartan element 
$m_+ \in \rhoc[1] \hat{\otimes}_\ke R$ with $\no{m_+} <1$
and} we call such a triple $(\A, m, m_+)$ a 
{\it gapped filtered graded $\ainf$ category over $R$}. 
See \cite[Definition 3.2.26]{fooo1} for the definition of gappedness of a filtered 
$\ainf$ algebra over the universal Novikov ring, and see also \cite[Definition 3.7.5]{fooo1}.
\end{remark}
\begin{remark}\label{filcon}
We consider the same setting as in Remark \ref{filainf}.
Let $(\A, m, m_+)$ be a gapped filtered graded $\ainf$ category over $R$ as above.
Then the trivial connection $\nabla$ on the normed DGLA $\rhoc[1] \widehat{\otimes}_{\ke} 
R$
is an $\LI$ connection and 
the trivial ``connection'' $\wt{\nabla}$ on the normed CH module 
$(\rhoh[1] \widehat{\otimes}_\ke R)[[z]]$ 
is a CH connection.
Here the normed DGLA structure and the normed CH structure are determined by $m$ (not twisted by the Maurer-Cartan element $m_+$).
Since $m_+$ is a Maurer-Cartan element of $\rhoc[1]\widehat{\otimes}_\ke R$, we can consider the Euler connection
on the cohomology of the ($m_+$-twisted) normed CH module 
$(\rhoh[1] \widehat{\otimes}_\ke R)[[z]]^{m_+}$.
See Subsection \ref{MC} for this notation. 
In a way similar to \cite{getcar}, we can prove flatness of $\nabla^\ggm$ on cohomology and 
also flatness of $\nabla^\eul$ by Proposition \ref{flat}.
\end{remark}

\subsection{Proofs of Propositions \ref{ch1} and \ref{ch2}}\label{proof} 
Propositions \ref{ch1} and \ref{ch2} follow by direct calculation.
To write down the proofs, we use operadic notation.
In \cite{konnot}, they introduce a 2-colored operad called the Kontsevich-Soibelman operad
and show that this operad naturally acts on Hochschild invariants
(see also \cite{dolfor}, \cite{wilhom}).
In this article, we will follow \cite[\S 3]{wilhom}.
They use ``trees'' to describe the operadic structure. 
We use parentheses notation to describe the operadic structure.
We note that there is a one to one correspondence between trees and parentheses.
To simplify notation, the special vertex $\vout$ is omitted, i.e.,$\{\vout\{\cdots\}\}$ is denoted by $\{\cdots\}$ (see also Example \ref{6.9} below). 
\begin{remark}
Precisely, we consider the degree one shift of the Kontsevich-Soibelman operad which acts on 
the pair $(\rhoc[1], \rhoh[1]).$
The $\rhoc[1]$ part of this structure is closely related to the brace algebra \cite{gerhom}. 
\end{remark}
\begin{example}\label{6.9}
\begin{align*}
\varphi \circ\psi=\varphi\{\psi\},\ 
\LL_\varphi=\{\vin, \varphi\}+\{\varphi\{\vin\}\},\ 
\rho_{\varphi,\psi}=\{\varphi\{\psi, \vin\}\},\ 
B=\{\mathbbm{1}, \vin \},\ 
B^1_\varphi=\{\mathbbm{1}, \varphi, \vin\}.
\end{align*}
\end{example}
\par\noindent
{\it Proof of Proposition \ref{ch1}.} 
Proposition $\ref{ch1}$ easily follows from the next lemma.
\begin{lemma}
\begin{align*}
 \rho_{\varphi_1 \circ \varphi_2, \psi}
 -\rho_{\varphi_1, [\varphi_2, \psi]}
 +(-1)^{|\varphi_1|'|\varphi_2|'} [\LL_{\varphi_2}, \rho_{\varphi_1, \psi}]
 =\{\varphi_1\{\varphi_2\{\psi, \vin\}\}\}
 +(-1)^{|\varphi_1|'|\varphi_2|'}\{\varphi_2\{\varphi_1\{\psi, \vin\}\}\}
\end{align*}
\end{lemma}
\begin{proof}
By the definition, we have 
\begin{align*}
 \rho_{\varphi_1, \varphi_2 \circ \psi}=\{\varphi_1\{\varphi_2\{\psi\}, \vin\}\}, \ \ 
 \rho_{\varphi_1, \psi \circ \varphi_2}=\{\varphi_1\{\psi\{\varphi_2\}, \vin\}\}.
\end{align*}
By direct calculation, we see that
\begin{align*}
 \rho_{\varphi_1 \circ \varphi_2, \psi}
&=\{\varphi_1\{\varphi_2, \psi, \vin\}\}
 +(-1)^{|\varphi_2|'|\psi|'}(\{\varphi_1\{\psi, \varphi_2, \vin\}\}
 +\{\varphi_1\{\psi, \vin, \varphi_2\}\}
 +\{\varphi_1\{\psi, \varphi_2\{\vin\}\}\}) \\
&+\{\varphi_1\{\varphi_2\{\psi\}, \vin\}\}
 +\{\varphi_1\{\varphi_2\{\psi, \vin\}\}\}. \\
 \rho_{\varphi_1, \psi} \circ \LL_{\varphi_2}
&=\{\varphi_1\{\psi, \vin\}, \varphi_2\}
  +\{\varphi_1\{\psi\{\varphi_2\}, \vin\}\} \\
&+\{\varphi_1\{\psi, \varphi_2, \vin\}\}
  +\{\varphi_1\{\psi, \vin, \varphi_2\}\}
  +(-1)^{|\varphi_2|'|\psi|'}\{\varphi_1\{\varphi_2, \psi, \vin\}\}
  +\{\varphi_1\{\psi, \varphi_2\{\vin\}\}\}. \\
\LL_{\varphi_2} \circ\rho_{\varphi_1, \psi}
&=(-1)^{|\varphi_1|'|\varphi_2|'+|\varphi_2|'|\psi|'}\{\varphi_1\{\psi, \vin\}, \varphi_2\}
  +\{\varphi_2\{\varphi_1\{\psi, \vin\}\}\}.
\end{align*}
This lemma follows from these equations.
Note that $\LL_{\varphi_2}$ is degree $|\varphi_2|'$ and $\rho_{\varphi_1, \psi}$ is degree 
$|\varphi_1|'+|\psi|'.$
We also note that in this calculation we do not need to assume $\varphi_1, \varphi_2, \psi$ are reduced cochains. 
\end{proof}
We next prove Proposition \ref{ch2}.
\begin{proof}[Proof of Proposition \ref{ch2}]
By the definition, we have 
\begin{align*}
[B, \rho_{\varphi, \psi}]=B \circ \rho_{\varphi, \psi}=\{\mathbbm{1}, \varphi\{\psi, \vin\}\},\ \ 
B^1_{\varphi \circ \psi}=\{\mathbbm{1}, \varphi\{\psi\}, \vin\},\ \ 
B^1_{\psi \circ \varphi}=\{\mathbbm{1}, \psi\{\varphi\}, \vin\}.
\end{align*}
By direct calculation, we have
\begin{align*}
B^1_\psi \circ \LL_\varphi
&=\{ \mathbbm{1},  \psi, \varphi \{\vin\} \}  \\
&+\{ \mathbbm{1},  \psi, \vin, \varphi \} 
  +\{ \mathbbm{1},  \psi, \varphi, \vin \}
  +(-1)^{|\varphi|'|\psi|'} \{ \mathbbm{1}, \varphi, \psi, \vin\}
  +\{ \mathbbm{1},  \psi \{ \varphi \}, \vin \},\\
\LL_\varphi \circ B^1_\psi
&=(-1)^{|\varphi|'}(\{ \mathbbm{1},  \varphi, \psi, \vin\}
  +\{\mathbbm{1}, \varphi\{\psi, \vin\}\}
  +\{\mathbbm{1}, \varphi\{\psi\}, \vin\}) \\
&+(-1)^{|\varphi|'+|\varphi|'|\psi|'}(\{ \mathbbm{1},  \psi, \varphi \{\vin\} \} 
  +\{ \mathbbm{1}, \psi, \vin, \varphi \} 
  +\{ \mathbbm{1}, \psi, \varphi, \vin \}).  
\end{align*}
Since $B^1_\psi$ is degree $|\psi|'+1$, we have 
\[ (-1)^{|\varphi|'}[\LL_\varphi, B^1_\psi]
  =\{\mathbbm{1}, \varphi\{\psi, \vin\}\}
  +\{\mathbbm{1}, \varphi\{\psi\}, \vin\}
  -(-1)^{|\varphi|'|\psi|'} \{ \mathbbm{1},  \psi \{ \varphi \}, \vin \}.\]
The proposition follows from these equations.  
\end{proof}


\section{Primitive forms without higher residue pairings}\label{prim}
In this section, we briefly recall a part of the definition of the primitive forms \cite{saiper}.
Let $\ke$ be a field (ungraded) and $R$ be a graded algebra over the graded ring $\ke^e :=\ke((e))$.
Let $\wt{M}$ be a finitely generated free graded $R[[z]]$ module equipped with an integrable (flat) meromorphic connection 
\[\nabla : \der_\ke \big{(}R((z))\big{)} \to 
 \End_\ke \big{(}\wt{M}((z)) \big{)},\]
where $\wt{M}((z)):=\wt{M}\otimes_{R[[z]]}R((z))$. 
We assume that $\nabla$ has pole order at most $2$ in the $z$-direction and pole order at most $1$ in another direction 
(i.e., $z\nabla_X m \in \wt{M}$ for $X \in \der_\ke(R)$ and $m \in \wt{M}$).
An element $\zeta \in \wt{M}^r$ is said to be {\it primitive} if the morphism 
\[ [z\nabla\zeta] : \der_{\ke^e}(R) \to (\wt{M}/z\wt{M})[r+2]\]
is an isomorphism as $R$ modules.
If $\zeta \in \wt{M}^r$ is primitive, then $z\nabla\zeta : \der_{\ke^e} (R)[[z]] \to \wt{M}[r+2]$
is also an isomorphism as $R[[z]]$ modules.  
By this isomorphism, $\nabla$ gives a meromorphic connection on $\der_{\ke^e}(R)((z))$
which is also denoted by $\nabla.$
For the Euler vector field $E \in \der^0_\ke(R)$, 
there exists a unique element $E' \in \der^0_{\ke^e}(R)$
such that $[z\nabla_E \zeta]=[z \nabla_{E'} \zeta].$
By abuse of  notation, this vector field $E'$ is also denoted by $E$ and called an Euler vector field.
\begin{definition}[{cf. \cite[Definition 6.1, (P1), (P3), (P4), (P5)]{saifro}}]
An element $\zeta \in \wt{M}^r$ is called a {\it primitive form without higher residue pairing} if 
\begin{itemize}
\item $\zeta$ is primitive.
\item $z \nabla_\dz\zeta=-\nabla_E\zeta+\frac{r}{2}\zeta$.
\item $z\nabla_XY \in \der_{\ke^e}(R) \oplus z\der_{\ke^e}(R)$ for all $X, Y \in \der_{\ke^e}(R).$
\item $z^2\nabla_\frac{d}{dz}Y \in  \der_{\ke^e}(R) \oplus z\der_{\ke^e}(R) \oplus z^2\der_{\ke^e}(R)$ 
          for all $Y \in \der_{\ke^e}(R).$
\end{itemize} 
\end{definition}
\begin{remark}
In this note, there is no specific reason to exclude the parts related to higher residue parings 
from the definition of the primitive forms.
\end{remark}
\begin{remark}({\bf Continued from Remarks \ref{filainf} and \ref{filcon}}.)
Let $(\A, m, m_+)$ be a gapped filtered graded $\ainf$ category over $R$
as in Remark \ref{filainf}. 
We define a graded $R[[z]]$ module $\wt{M}$ by the cohomology of the normed CH module $\rhoh[1]\widehat{\otimes}_\ke R)[[z]]^{m_+}.$
The graded $R[[z]]$ module $\wt{M}$ is equipped with the Euler connection $\nabla^\eul$.
   If $\wt{M}$ is finitely generated and free over $R[[z]]$, then we can consider primitive forms without higher residue pairings on $(\wt{M}, \nabla^\eul)$.
We note that, in this categorical setting, primitive forms are closely related to (smooth or proper) Calabi-Yau structures (see, e.g., \cite{ganmir}).   
\end{remark}
\begin{example}
({\bf Continued from Example \ref{5.20}}.)
An element $\mathbbm{1}$ in $A^0$ is called a unit $($of the hyper commutative algebra $A^e$$)$
if \[(\mathbbm{1}, x_1, \dots, x_k)=\begin{cases} x_1 \ &k=1, \\ 0 \ &k \ge 2.  \end{cases}\]
Suppose that $A^e$ has a unit $\mathbbm{1}$, then $\mathbbm{1}$ is also a unit of $A^e[[t]].$
Since $z\nabla^\eul_{\frac{\partial}{\partial t_i}} \mathbbm{1}=T_i$, we see that $\mathbbm{1}$ is primitive $($with respect to the Euler connection$)$ and ${\frac{\partial}{\partial t_i}}$ is identified with $T_i.$
Hence $\mathbbm{1}$ is a primitive form without a higher residue pairing.
\end{example}


\begin{thebibliography}{10}

\bibitem{afooo}
M.~{Abouzaid}, K.~{Fukaya}, Y.-G. {Oh}, H.~{Ohta}, and K.~{Ono}.
\newblock Quantum cohomology and split generation in {L}agrangian {F}loer
  theory.
\newblock in preparation.

\bibitem{barqua}
S.~Barannikov.
\newblock Quantum periods, {I}: semi-infinite variations of hodge structures.
\newblock {\em International Mathematics Research Notices},
  2001(23):1243--1264, 2001.

\bibitem{bgr}
S.~Bosch, U.~G\"{u}ntzer, and R.~Remmert.
\newblock {\em Non-{A}rchimedean analysis}, volume 261 of {\em Grundlehren der
  Mathematischen Wissenschaften [Fundamental Principles of Mathematical
  Sciences]}.
\newblock Springer-Verlag, Berlin, 1984.
\newblock A systematic approach to rigid analytic geometry.

\bibitem{calcat}
A.~Caldararu, S.~Li, and J.~Tu.
\newblock {Categorical primitive forms and Gromov-Witten invariants of $A_n$
  singularities}.
\newblock {\em arXiv e-prints}, page arXiv:1810.05179, October 2018.

\bibitem{catcha}
A.~Cattaneo, G.~Felder, and T.~Willwacher.
\newblock The character map in deformation quantization.
\newblock {\em Adv. Math.}, 228(4):1966--1989, 2011.

\bibitem{dalope}
Y.~Daletski and B.~Tsygan.
\newblock Operations on {H}ochschild and cyclic complexes.
\newblock {\em Methods Funct. Anal. Topology}, 5(4):62--86, 1999.

\bibitem{dolfor}
V.~{Dolgushev}, D.~{Tamarkin}, and B.~{Tsygan}.
\newblock {Formality of the homotopy calculus algebra of Hochschild
  (co)chains}.
\newblock {\em arXiv e-prints}, page arXiv:0807.5117, July 2008.

\bibitem{dolnon}
V.~Dolgushev, D.~Tamarkin, and B.~Tsygan.
\newblock Noncommutative calculus and the {G}auss-{M}anin connection.
\newblock In {\em Higher structures in geometry and physics}, volume 287 of
  {\em Progr. Math.}, pages 139--158. Birkh\"{a}user/Springer, New York, 2011.

\bibitem{dubint}
B.~Dubrovin.
\newblock Integrable systems in topological field theory.
\newblock {\em Nuclear Phys. B}, 379(3):627--689, 1992.

\bibitem{dubgeo}
B.~Dubrovin.
\newblock Geometry of {$2$}{D} topological field theories.
\newblock In {\em Integrable systems and quantum groups ({M}ontecatini {T}erme,
  1993)}, volume 1620 of {\em Lecture Notes in Math.}, pages 120--348.
  Springer, Berlin, 1996.

\bibitem{fukcyc}
K.~Fukaya.
\newblock Cyclic symmetry and adic convergence in {L}agrangian {F}loer theory.
\newblock {\em Kyoto J. Math.}, 50(3):521--590, 2010.

\bibitem{fukuno}
K~{Fukaya}.
\newblock {Unobstructed immersed Lagrangian correspondence and filtered A
  infinity functor}.
\newblock {\em arXiv e-prints}, page arXiv:1706.02131, Jun 2017.

\bibitem{fooo00}
K.~{Fukaya}, Y.-G. {Oh}, H.~{Ohta}, and K~{Ono}.
\newblock Lagrangian intersection {F}loer theory: anomaly and obstruction.
\newblock {\em Kyoto University preprint}, 2000.
\newblock Available at http://www.math.kyoto-u.ac.jp/~fukaya/fukaya.html.

\bibitem{fooo1}
K.~{Fukaya}, Y.-G. {Oh}, H.~{Ohta}, and K.~{Ono}.
\newblock {\em Lagrangian intersection {F}loer theory: anomaly and obstruction.
  {P}art {I}}, volume 46-1 of {\em AMS/IP Studies in Advanced Mathematics}.
\newblock American Mathematical Society, Providence, RI, 2009.

\bibitem{foooanch}
K.~{Fukaya}, Y.-G. {Oh}, H.~{Ohta}, and K.~{Ono}.
\newblock Anchored {L}agrangian submanifolds and their {F}loer theory.
\newblock In {\em Mirror symmetry and tropical geometry}, volume 527 of {\em
  Contemp. Math.}, pages 15--54. Amer. Math. Soc., Providence, RI, 2010.

\bibitem{foootoric1}
K.~{Fukaya}, Y.-G. {Oh}, H.~{Ohta}, and K.~{Ono}.
\newblock Lagrangian {F}loer theory on compact toric manifolds {I}.
\newblock {\em Duke Math. J.}, 151(1):23--174, 2010.

\bibitem{foootoric2}
K.~{Fukaya}, Y.-G. {Oh}, H.~{Ohta}, and K.~{Ono}.
\newblock Lagrangian {F}loer theory on compact toric manifolds {II}: bulk
  deformations.
\newblock {\em Selecta Mathematica, New Series}, 17:609--711, 2011.

\bibitem{foootoricmirror}
K.~{Fukaya}, Y.-G. {Oh}, H.~{Ohta}, and K.~{Ono}.
\newblock Lagrangian {F}loer theory and mirror symmetry on compact toric
  manifolds.
\newblock {\em Ast\'erisque}, (376):vi+340, 2016.

\bibitem{ganmir}
S.~{Ganatra}, T.~{Perutz}, and N.~{Sheridan}.
\newblock {Mirror symmetry: from categories to curve counts}.
\newblock {\em ArXiv e-prints}, October 2015.

\bibitem{gerhom}
M.~Gerstenhaber and A.~Voronov.
\newblock Homotopy {G}-algebras and moduli space operad.
\newblock {\em International Mathematics Research Notices}, 1995(3):141--153,
  1995.

\bibitem{getcar}
E.~Getzler.
\newblock Cartan homotopy formulas and the {G}auss-{M}anin connection in cyclic
  homology.
\newblock In {\em Quantum deformations of algebras and their representations
  ({R}amat-{G}an, 1991/1992; {R}ehovot, 1991/1992)}, volume~7 of {\em Israel
  Math. Conf. Proc.}, pages 65--78. Bar-Ilan Univ., Ramat Gan, 1993.

\bibitem{getope}
E.~Getzler.
\newblock Operads and moduli spaces of genus 0 riemann surfaces.
\newblock In Robbert~H. Dijkgraaf, Carel~F. Faber, and Gerard B.~M. van~der
  Geer, editors, {\em The Moduli Space of Curves}, pages 199--230, Boston, MA,
  1995. Birkh{\"a}user Boston.

\bibitem{irimir}
H.~Iritani.
\newblock A mirror construction for the big equivariant quantum cohomology of
  toric manifolds.
\newblock {\em Mathematische Annalen}, 368(1):279--316, Jun 2017.

\bibitem{iwaper}
I.~Iwanari.
\newblock {Period mappings for noncommutative algebras}.
\newblock {\em arXiv e-prints}, page arXiv:1604.08283, Apr 2016.

\bibitem{kkp1}
L.~Katzarkov, M.~Kontsevich, and T.~Pantev.
\newblock Hodge theoretic aspects of mirror symmetry.
\newblock In {\em From {H}odge theory to integrability and {TQFT}
  tt*-geometry}, volume~78 of {\em Proc. Sympos. Pure Math.}, pages 87--174.
  Amer. Math. Soc., Providence, RI, 2008.

\bibitem{konnot}
M.~Kontsevich and Y.~Soibelman.
\newblock Notes on {$A_\infty$}-algebras, {$A_\infty$}-categories and
  non-commutative geometry.
\newblock In {\em Homological mirror symmetry}, volume 757 of {\em Lecture
  Notes in Phys.}, pages 153--219. Springer, Berlin, 2009.

\bibitem{ladint}
T.~Lada and J.~Stasheff.
\newblock Introduction to sh lie algebras for physicists.
\newblock {\em International Journal of Theoretical Physics}, 32(7):1087--1103,
  Jul 1993.

\bibitem{saipri}
C.~{Li}, S.~{Li}, and K.~{Saito}.
\newblock {Primitive forms via polyvector fields}.
\newblock {\em arXiv e-prints}, page arXiv:1311.1659, Nov 2013.

\bibitem{lodalg}
J.~Loday and B.~Vallette.
\newblock {\em Algebraic operads}, volume 346 of {\em Grundlehren der
  Mathematischen Wissenschaften [Fundamental Principles of Mathematical
  Sciences]}.
\newblock Springer, Heidelberg, 2012.

\bibitem{saiper}
K.~Saito.
\newblock Period mapping associated to a primitive form.
\newblock {\em Publ.RIMS}, 19:1231--1264, 1983.

\bibitem{saifro}
K.~Saito and A.~Takahashi.
\newblock From primitive forms to {F}robenius manifolds.
\newblock In {\em From {H}odge theory to integrability and {TQFT}
  tt*-geometry}, volume~78 of {\em Proc. Sympos. Pure Math.}, pages 31--48.
  Amer. Math. Soc., Providence, RI, 2008.

\bibitem{gradL}
P.~Seidel.
\newblock Graded {L}agrangian submanifolds.
\newblock {\em Bull. Soc. Math. France}, 128(1):103--149, 2000.

\bibitem{seicon}
P.~Seidel.
\newblock Connections on equivariant {H}amiltonian {F}loer cohomology.
\newblock {\em Comment. Math. Helv.}, 93(3):587--644, 2018.

\bibitem{shknon}
D.~Shklyarov.
\newblock Non-commutative {H}odge structures: towards matching categorical and
  geometric examples.
\newblock {\em Trans. Amer. Math. Soc.}, 366(6):2923--2974, 2014.

\bibitem{takcal}
A.~{Takahashi}.
\newblock {From Calabi-Yau dg Categories to Frobenius manifolds via Primitive
  Forms}.
\newblock {\em arXiv e-prints}, page arXiv:1503.09099, Mar 2015.

\bibitem{tamnon}
D.~Tamarkin and B.~Tsygan.
\newblock Noncommutative differential calculus, homotopy {BV} algebras and
  formality conjectures.
\newblock {\em Methods Funct. Anal. Topology}, 6(2):85--100, 2000.

\bibitem{tsygau}
B.~Tsygan.
\newblock On the {G}auss-{M}anin connection in cyclic homology.
\newblock {\em Methods Funct. Anal. Topology}, 13(1):83--94, 2007.

\bibitem{wilhom}
T.~Willwacher.
\newblock The homotopy braces formality morphism.
\newblock {\em Duke Math. J.}, 165(10):1815--1964, 07 2016.

\end{thebibliography}

\end{document}